\documentclass[12pt]{amsart}
\pdfoutput=1
\usepackage{latexsym, amsmath, amssymb, amsthm, mathrsfs}
\usepackage[mathscr]{eucal}
\usepackage{paralist}
\usepackage{graphicx}
\usepackage[colorlinks=true,linkcolor=blue,urlcolor=blue, citecolor=blue]%
  {hyperref}
\usepackage[shortalphabetic]{amsrefs}
\usepackage[T1]{fontenc}
\usepackage{mathptmx}
\usepackage{microtype}

\usepackage[centering, includeheadfoot, hmargin=1.0in, tmargin=1.0in, 
  bmargin=1in, headheight=30.4pt]{geometry}
\newtheorem{lemma}{Lemma}[section]
\newtheorem{theorem}[lemma]{Theorem}
\newtheorem{corollary}[lemma]{Corollary}
\newtheorem{proposition}[lemma]{Proposition}
\theoremstyle{definition}
\newtheorem{definition}[lemma]{Definition}
\newtheorem{remark}[lemma]{Remark}
\newtheorem{example}[lemma]{Example}
\newcommand{\define}[1]{{\bfseries\itshape #1}}

\newcommand{\abs}[1]{\ensuremath{\left| #1 \right|}}
\newcommand{\norm}[1]{\ensuremath{\| #1 \|}}
\newcommand{\genus}{\ensuremath{p_{\text{\upshape a}}}}
\newcommand{\coloneq}{\ensuremath{\mathrel{\mathop :}=}}
\renewcommand{\AA}{\ensuremath{\mathbb{A}}} 
\newcommand{\CC}{\ensuremath{\mathbb{C}}} 
\newcommand{\NN}{\ensuremath{\mathbb{N}}}
\newcommand{\PP}{\ensuremath{\mathbb{P}}} 
\newcommand{\QQ}{\ensuremath{\mathbb{Q}}} 
\newcommand{\RR}{\ensuremath{\mathbb{R}}} 
\renewcommand{\SS}{\ensuremath{\mathbb{S}}} 
\newcommand{\ZZ}{\ensuremath{\mathbb{Z}}} 
\renewcommand{\geq}{\geqslant}
\renewcommand{\leq}{\leqslant}
\DeclareMathOperator{\Ann}{Ann}
\DeclareMathOperator{\Area}{area}
\DeclareMathOperator{\codim}{codim}

\DeclareMathOperator{\conv}{conv}
\DeclareMathOperator{\der}{d}
\DeclareMathOperator{\ee}{e}
\DeclareMathOperator{\Grass}{Gr}
\DeclareMathOperator{\HF}{h}
\DeclareMathOperator{\HP}{p}

\DeclareMathOperator{\Ker}{Ker}
\DeclareMathOperator{\Pos}{P}
\DeclareMathOperator{\Proj}{Proj}
\DeclareMathOperator{\rank}{rank}
\DeclareMathOperator{\rr}{r}
\DeclareMathOperator{\sign}{sgn}
\DeclareMathOperator{\Sos}{\Sigma}
\DeclareMathOperator{\Span}{Span}

\DeclareMathOperator{\Sym}{Sym}
\DeclareMathOperator{\transpose}{\textsf{T}}
\DeclareMathOperator{\variety}{V}

\begin{document}

%\vspace*{-1.0em}

\title[Degree Bounds for Sum-of-Squares Certificates]{Sharp Degree Bounds for
  Sum-of-Squares Certificates\\ on Projective Curves}

\author[G.~Blekherman]{Grigoriy Blekherman}
\address{Greg Blekherman: School of Mathematics \\ Georgia Tech, 686 Cherry
  Street\\ Atlanta\\ GA\\ 30332\\ USA; {\normalfont
    \href{mailto:greg@math.gatech.edu}{\texttt{greg@math.gatech.edu}}.}}
% \email{\href{mailto:greg@math.gatech.edu}{greg@math.gatech.edu}}

\author[G.G.~Smith]{Gregory G. Smith} 
\address{Gregory G. Smith: Department of Mathematics and Statistics \\
  Queen's University\\ Kingston \\ ON \\ K7L~3N6\\ Canada; {\normalfont
    \href{mailto:ggsmith@mast.queensu.ca}{\texttt{ggsmith@mast.queensu.ca}}.}}
% \email{\href{mailto:ggsmith@mast.queensu.ca}{ggsmith@mast.queensu.ca}}

\author[M.~Velasco]{Mauricio Velasco} 

\address{Mauricio Velasco: Departamento de Matem\'aticas\\ Universidad de los
  Andes\\ Carrera 1 No. 18a 10\\ Edificio H\\ Primer Piso\\ 111711 Bogot\'a\\
  Colombia; {\normalfont
    \href{mailto:mvelasco@uniandes.edu.co}%
    {\texttt{mvelasco@uniandes.edu.co}}.}}
% \textrm{and} Departamento de
% Matem\'aticas y Aplicaciones\\ Universidad de la Rep\'ublica (CURE)\\
% Maldonado\\ Uruguay
% }
%   \email{\href{mailto:mvelasco@uniandes.edu.co}{mvelasco@uniandes.edu.co} 
%   or
%   \href{mailto:mvelasco@cmat.edu.uy}{mvelasco@cmat.edu.uy}
% }

%\keywords{real algebraic geometry, sums of squares, special curves, surfaces
%  of minimal degree}


\date{24 June 2018}

\begin{abstract}
  Given a real projective curve with homogeneous coordinate ring $R$ and a
  nonnegative homogeneous element $f \in R$, we bound the degree of a nonzero
  homogeneous sum-of-squares $g \in R$ such that the product $fg$ is again a
  sum of squares. Better yet, our degree bounds only depend on geometric
  invariants of the curve and we show that there exist smooth curves and
  nonnegative elements for which our bounds are sharp.  We deduce the
  existence of a multiplier $g$ from a new Bertini Theorem in convex algebraic
  geometry and prove sharpness by deforming rational Harnack curves on toric
  surfaces.  Our techniques also yield similar bounds for multipliers on
  surfaces of minimal degree, generalizing Hilbert's work on ternary forms.
\end{abstract}

\maketitle

%\vspace*{-1.5em}

%%%%%%%%%%%%%%%%%%%%%%%%%%%%%%%%%%%%%%%%%%%%%%%%%%%%%%%%%%%%%%%%%%%%%%%%%%%%%%
\section{Overview of Results} 
\label{sec:intro}

\noindent
Certifying that a polynomial is nonnegative remains a central problem in real
algebraic geometry and optimization.  The quintessential certificate arises
from multiplying a given polynomial by a second polynomial, that is already
known to be positive, and expressing the product as a sum of squares.
Although the Positivstellensatz guarantees that suitable multipliers exist
over any semi-algebraic set, tight bounds on the degree of multipliers are
exceptionally rare.  Our primary aim is to produce sharp degree bounds for
sum-of-squares multipliers on real projective curves.  In reaching this goal,
the degree bounds also reveal a surprising consonance between real and complex
algebraic geometry.

To be more explicit, fix an embedded real projective curve $X \subset \PP^n$
that is nondegenerate and totally real---not contained in a hyperplane and
with Zariski-dense real points.  Let $R$ be its $\ZZ$-graded coordinate ring
and let $\rr(X)$ denote the least integer $i$ such that the Hilbert polynomial
and function of $X$ agree at all integers greater than or equal to $i$.  For
$j \in \NN$, we write $\Pos_{X,2j} \subset R_{2j}$ and
$\Sos_{X,2j} \subset R_{2j}$ for the cone of nonnegative elements in $R_{2j}$
and the cone of sums of squares of elements from $R_j$, respectively.  Our
first result gives a sharp degree bound on sum-of-squares multipliers in terms
of the fundamental geometric invariants of $X$.

\begin{theorem}
  \label{thm:main}
  For any nondegenerate totally-real projective curve $X \subset \PP^n$ of
  degree $d$ and arithmetic genus $\genus$, any nonnegative element
  $f \in \Pos_{X,2j}$ of positive degree, and all $k \in \NN$ satisfying
  $k \geq \max\bigl\{ \rr(X), \frac{2 \genus}{d} \bigr\}$, there is a nonzero
  $g \in \Sos_{X,2k}$ such that $fg \in \Sos_{X,2j+2k}$.  Conversely, for all
  $n \geq 2$ and all $j \geq 2$, there exist totally-real smooth curves
  $X \subset \PP^n$ and nonnegative elements $f \in \Pos_{X,2j}$ such that,
  for all $k < \max\bigl\{ \rr(X), \frac{2 \genus}{d} \bigr\}$ and for all
  nonzero $g \in \Sos_{X,2k}$, we have $fg \not\in \Sos_{X,2j+2k}$.
\end{theorem}

\noindent
Remarkably, the uniform degree bound on the multiplier $g$ is determined by
the complex geometry of the curve $X$. It is independent of both the degree of
the nonnegative element $f$ and the Euclidean topology of the real points in
$X$.

Our approach also applies to higher-dimensional varieties that are
arithmetically Cohen--Macaulay, but it is most effective on certain surfaces.
A subvariety $X \subset \PP^n$ has minimal degree if it is nondegenerate and
$\deg(X) = 1 + \codim(X)$.  Theorem~1.1 in \cite{BSV} establishes that
$\Pos_{X,2} = \Sos_{X,2}$ if and only if $X$ is a totally-real variety of
minimal degree.  Evocatively, this equivalence leads to a characterization of
the varieties over which multipliers of degree zero suffice.  Building on this
framework and generalizing Hilbert's work~\cite{Hilbert2} on ternary forms,
our second result gives degree bounds for sum-of-squares multipliers on
surfaces of minimal degree.

\begin{theorem}
  \label{thm:surface}
  If $X \subset \PP^n$ is a totally-real surface of minimal degree and
  $f \in \Pos_{X,2j}$ is a nonnegative element of positive degree, then there
  is a nonzero $g \in \Sos_{j^2 - j}$ such that $f g \in \Sos_{X,j^2+j}$.
  Conversely, if $X \subset \PP^n$ is a totally-real surface of minimal degree
  and $j \geq 2$, then there exist nonnegative elements $f \in \Pos_{X,2j}$
  such that, for all $k < j-2$ and all nonzero $g \in \Sos_{X,2k}$, we have
  $fg \not\in \Sos_{X,2j+2k}$.
\end{theorem}

\noindent
Unlike curves, Theorem~\ref{thm:surface} shows that the minimum degree of a
sum-of-squares multipler $g$ depends intrinsically on the degree of the
nonnegative element $f$.  The sharpness of the upper or lower bounds on
these surfaces is an intriguing open problem.

Motivated by its relation to Hilbert's Seventeenth Problem, we obtain slightly
better degree bounds when the totally-real surface is $\PP^2$; see
Example~\ref{exa:formsOnPP2} and Example~\ref{exa:strictPP^2} for the details.
Specifically, we re-prove and prove the following two results for ternary
octics:
\begin{itemize}[$\bullet$]
\item for all nonnegative $f \in \Pos_{\PP^2,8}$, there exists a nonzero
  $g \in \Sos_{\PP^2, 4}$ such that $f g \in \Sos_{\PP^2, 12}$; and
\item there exists a nonnegative $f \in \Pos_{\PP^2,8}$ such that, for all
  nonzero $g \in \Sos_{\PP^2, 2}$, we have $f g \not\in \Sos_{\PP^2, 10}$.
\end{itemize}
Together these give the first tight bounds on the degrees of sum-of-squares
multipliers for homogeneous polynomials since Hilbert's 1893
paper~\cite{Hilbert2} in which he proves sharp bounds for ternary sextics. No
other sharp bounds for homogeneous polynomials are known.  For example, the
recent theorem in \cite{Pas} shows that, for quaternary quartics, one can
multiply by sum-of-squares of degree $4$ to obtain a sum of squares, but it is
not known whether quadratic multipliers suffice.

By reinterpreting Theorem~\ref{thm:main} or Theorem~\ref{thm:surface}, we do
obtain degree bounds for certificates of nonnegativity.  A sum-of-squares
multiplier $g$ certifies that the element $f$ is nonnegative at all points
where $g$ does not vanish.  When the complement of this vanishing set is dense
in the Euclidean topology, it follows that the element $f$ is nonnegative.
Changing perspectives, these theorems also generate a finite hierarchy of
approximations to the cone $\Pos_{X,2j}$, namely the sets
$\{ f \in R_{2j} : \text{there exists $g \in \Sos_{X,2k}$ such that
  $fg \in \Sos_{X,2j+2k}$} \}$; compare with Subsection~3.6.1 in \cite{BPT}.
It follows that deciding if an element $f$ belongs to the cone $\Pos_{X,2j}$
is determined by a semidefinite program of known size.

%%----------------------------------------------------------------------------
\subsection*{Relationship with prior results}

Our degree bounds, with their uniformity and sharpness, cannot be directly
compared to any established bound on multipliers, except for those on
zero-dimensional schemes in \cite{BGP}.  Most earlier work focuses on general
semi-algebraic sets, where no sharpness results are known, or on affine
curves, where no uniform bounds are possible for singular curves.

The best bound on the degree of a sum-of-squares multiplier on an arbitrary
semi-algebraic set involves a tower of five exponentials; see Theorem~1.5.7 in
\cite{LPR}.  However, Corollary~\ref{cor:degreeOnly} shows that, for a
nondegenerate totally-real projective curve $X \subset \PP^n$ of degree $d$,
every nonnegative form admits a nonzero sum-of-squares multiplier of degree
$2k$ for all $k \geq d - n +1$.  Absent sharp bounds in some larger context, it
impossible to ascertain if this difference in the complexity of the bounds is
just a feature of low-dimensional varieties or part of some more general
phenomenon.

Restricting to curves likewise fails to produce meaningful
comparisons. Corollary~4.15 in \cite{ScheidererC} illustrates that one can
often certify nonnegativity without using a multiplier on an affine curve.
Concentrating on a special type of multiplier, Theorem~4.11 in
\cite{ScheidererP} demonstrates that, on a nonsingular projective curve, any
sufficiently large power of a positive element gives a multiplier; also see
\cite{Rez}.  For nonsingular affine curves, Corollary~4.4 in
\cite{ScheidererS} establishes that there exist uniform degree bounds, even
though the techniques do not yield explicit results. In contrast with
Theorem~\ref{thm:main}, the bounds in these situations either depend on the
nonnegative element $f$ or tend towards positive infinite as the underlying
curve acquires certain singularities.

To identify a close analogue of our work, we must lower the dimension:
Theorems~1.1--1.2 in \cite{BGP} provide uniform degree bounds over a finite
set of points that are tight for quadratic functions on the hypercube.  The
lone additional sharp degree bound on multipliers is, to the best of our
knowledge, Hilbert's original work~\cite{Hilbert2} on ternary sextics.

%%----------------------------------------------------------------------------
\subsection*{Main ideas}

The results in this paper arose while exploring the relationship between
convex geometry and algebraic geometry for sums of squares on real varieties.
The two parts of our main theorems are proven independently. The upper bound
on the minimum degree of a sum-of-squares multiplier is derived from a new
Bertini Theorem in convex algebraic geometry and the lower bound is obtained
by deforming rational Harnack curves on toric surfaces.

To prove the first part, we reinterpret the non-existence of a sum-of-squares
multiplier $g \in \Sos_{X,2k}$ as asserting that the convex cones
$\Sos_{X,2j+2k}$ and $f \cdot \Sos_{X,2k}$ intersect only at zero.  If a real
subscheme $X \subseteq \PP^n$ possesses a linear functional separating these
cones, then Theorem~\ref{thm:Bertini} demonstrates that a sufficiently general
hypersurface section of $X$ also does.  In this setting, the phrase
`sufficiently general' means belonging to a nonempty open subset in the
Euclidean topology of the relevant parameter space.  Unexpectedly, this convex
version of Bertini's Theorem relies on our characterization of spectrahedral
cones that have many facets in a neighbourhood of every point; see
Proposition~\ref{pro:pointed}.  Recognizing this dependency is the crucial
insight.  By repeated applications of our Bertini Theorem, we reduce to the
case of points.  Theorem~\ref{thm:upperBound} establishes the degree bound for
the existence of a sum-of-squares multiplier on curves and
Theorem~\ref{thm:totallyreal} gives a higher-dimensional variant on
arithmetically Cohen--Macaulay varieties.

To prove the second part, we show that having a nonnegative element vanish at a
relatively large number of isolated real singularities precludes it from
having a low-degree sum-of-squares multiplier.  As Proposition~\ref{pro:lower}
indicates, the hypotheses needed to actually realize this basic premise are
formidable.  Nonetheless, this transforms the problem into finding enough
curves that satisfy the conditions and maximize the number of isolated real
singularities.  Proposition~\ref{pro:Harnack} confirms that rational singular
Harnack curves on toric surfaces fulfill these requirements.  By perturbing
both the curve $X$ and the nonnegative element $f \in R_{2j}$,
Theorem~\ref{thm:lower} exhibits smooth curves and nonnegative elements
without low-degree sum-of-squares multipliers.  Proposition~\ref{pro:lifting}
then lifts these degree bounds from curves to some surfaces.  Miraculously,
for totally-real projective curves, the degree bounds in
Theorem~\ref{thm:upperBound} and Theorem~\ref{thm:lower} coincide.

%%----------------------------------------------------------------------------
\subsection*{Explicit Examples}

Beyond the uniformity, our results also specialize to simple degree bounds in
many interesting situations.  As one might expect, the degree bounds are
straightforward for complete intersections and planar curves; see
Example~\ref{exa:CI} and Example~\ref{exa:planar}.  However, by demonstrating
that our degree bound is sharp for some, but not all, planar curves,
Example~\ref{exa:nonoptimal} and Example~\ref{exa:deltoid} are much more
innovative.  For certain non-planar curves lying on embedded toric surfaces,
Examples~\ref{exa:sharp}--\ref{exa:Veronese} present sharp degree bounds.
These examples also serve as our best justification for the second part of
Theorem~\ref{thm:main}.  It remains an open problem to classify all of the
curves for which the bounds in Theorem~\ref{thm:main} are sharp. Switching to
higher-dimensional arithmetically Cohen--Macaulay varieties,
Example~\ref{exa:quadraticForms} re-establishes that a nonnegative quadratic
form on a totally-real variety of minimal degree is a sum of
squares. Examples~\ref{exa:minimalSurface}--\ref{exa:almost}, which bound
multipliers on surfaces of minimal degree, the projective plane, and surfaces
of almost minimal degree respectively, appear to exhaust all of the
consequential applications to surfaces.  Together
Example~\ref{exa:minimalSurface} and Example~\ref{exa:strictMinimalDegree}
establish Theorem~\ref{thm:surface}. Highlighting the peculiarity of curves,
this pair of examples also illustrates the gap between our upper and lower
bounds on the minimal degree of a multiplier in this case.  Nevertheless,
Example~\ref{exa:strictPP^2} does give our new sharp bound for ternary octics.
Despite being labelled examples, these are essential aspects of the paper.

%%%%%%%%%%%%%%%%%%%%%%%%%%%%%%%%%%%%%%%%%%%%%%%%%%%%%%%%%%%%%%%%%%%%%%%%%%%%%%
\section{Many-Faceted Spectrahedral Cones}
\label{sec:many-faceted}

\noindent
This section focuses on convex geometry and properties of spectrahedral cones.
We distinguish certain spectrahedral cones that have an abundance of facets in
the vicinity of every point.  To demonstrate the ubiquity of these cones in
convex algebraic geometry, we show that if a sum-of-squares cone is closed and
contains no lines, then its dual has this structure.

Let $V$ be a finite-dimensional real vector space, let
$S_2 \coloneq \Sym^2(V^*)$ be the vector space of quadratic forms on $V$, and
let $S_2^+ \subseteq S_2$ be the cone of positive-semidefinite quadratic
forms.  The corank of a quadratic form $f \in S_2$ is the dimension of the
kernel $\Ker(f)$ of the associated symmetric matrix.  We endow $\Sym^2(V^*)$
with the metric topology arising from the spectral norm.  Since all norms on a
finite-dimensional vector space induce the same topology, we refer to this
metric topology as the Euclidean topology.  For a quadratic form $g \in S_2$
and a positive real number $\varepsilon$, we write
$B_{\varepsilon}(g) \subset S_2$ for the open ball of radius $\varepsilon$
centered at $g$.  As usual, we equip each subset $W \subseteq S_2$ with the
induced Euclidean topology and the boundary $\partial W$ equals the closure of
$W$ in its affine span without the interior of $W$.

A linear subspace $L \subseteq S_2$ determines a spectrahedral cone
$C \coloneq L \cap S_2^+$.  The faces of the convex set $C$ have a useful algebraic
description.  Specifically, Theorem~1 in \cite{RamanaGoldman} establishes that
the minimal face of $C$ containing a given quadratic form $g \in S_2$ equals
the intersection of $C$ with the linear subspace consisting of $f \in S_2$
such that $\Ker(g) \subseteq \Ker(f)$.  Hence, if the linear subspace $L$
intersects the interior of the cone $S_2^+$, then a quadratic form having
corank $1$ determines a facet, that is an inclusion-maximal proper face.

Our first lemma identifies a special type of spectrahedral cone.  Given a
nonzero $v \in V$, let $T_v \subset S_2$ denote the linear subspace consisting
of the quadratic forms $f \in S_2$ such that $v \in \Ker(f)$.

\begin{lemma}
  \label{lem:wellDefined}
  If the quadratic form $g \in \partial C$ has corank $1$ and
  $\varepsilon > 0$ is sufficiently small, then the map
  $\varphi \colon B_{\varepsilon}(g) \cap \partial C \to \PP(V)$, sending a
  quadratic form $f$ to the linear subspace $\Ker(f)$, is well-defined.
  Moreover, when the defining linear subspace $L$ intersects the linear
  subspace $T_{\varphi(g)}$ transversely, the image of $\varphi$ contains a
  neighbourhood of $\varphi(g)$.
\end{lemma}

\begin{proof}
  The existence of $g \in \partial C$ having corank $1$ implies that the
  defining linear subspace $L$ meets the interior of $S_2^+$.  If not, then
  $C = L \cap S_2^+$ would be entirely contained in a face of $S_2^+$, and all
  of its boundary points would have corank at least $2$.

  We claim that $\partial C = L \cap \partial S_2^+$.  Since every
  neighbourhood of a point in $\partial C$ contains at least one point in
  $S_2^+$ and at least one point not in $S_2^+$, we have
  $\partial C \subseteq L \cap \partial S_2^+$.  On the other hand, suppose
  that $f \in L \cap \partial S_2^+$ belongs to the relative interior of $C$.
  Since $f \in \partial S_2^+$, there exists a nonzero linear functional
  $\ell \in S_2^*$ that is nonnegative on $S_2^+$ and satisfies $\ell(f) = 0$.
  As $f$ lies in the relative interior of $C$, it follows that $\ell$ vanishes
  identically on $C$ and the cone $C$ is contained in
  $\ell^{-1}(0) \cap S_2^+ \subseteq \partial S_2^+$.  However, this is absurd
  because $L$ intersects the interior of $S_2^+$, so we obtain
  $\partial C = L\cap \partial S_2^+$.

  Since the eigenvalues of a matrix are continuous functions of the entries of
  the matrix, we see that, for a sufficiently small $\varepsilon > 0$, each
  point in $B_{\varepsilon}(g) \cap \partial S_2^+$ has corank $1$.  Hence,
  for each quadratic form $f \in B_{\varepsilon}(g) \cap \partial C$, the
  linear subspace $\Ker(f)$ has dimension one.  Therefore, the map
  $\varphi \colon B_{\varepsilon}(g) \cap \partial C \to \PP(V)$ is
  well-defined.

  To prove the second part of the lemma, let $m \coloneq \dim(V)$.  By definition,
  we have $T_{\lambda v} = T_v$ for any nonzero $\lambda \in \RR$.  Moreover,
  requiring that a nonzero vector $v \in V$ belongs to the kernel of a
  symmetric $(m \times m)$-matrix imposes $m$ independent linear conditions on
  the entries of the matrix, so we have $\codim T_v = m$.  The hypothesis that
  the linear subspaces $L$ and $T_{\varphi(g)}$ meet transversely means that
  $S_2 = L + T_{\varphi(g)}$, which implies that
  $\dim S_2 = \dim (L) + \dim (T_{\varphi(g)}) - \dim (L \cap T_{\varphi(g)})$
  and $\dim(L \cap T_{\varphi(g)}) = \dim(L) - m$.  It follows that, for all
  $v$ in an open neighbourhood $W \subseteq \PP(V)$ of $\varphi(g)$ in the
  Euclidean topology, we have $\dim(L \cap T_{v}) = \dim(L) - m$.

  Consider the set
  $U_{\varepsilon} \coloneq \bigl\{ [v] \in W : T_v \cap L \cap B_\epsilon(g)
  \neq \varnothing \bigr\}$.  Let $G \coloneq \Grass\bigl(\dim(L)-m, L \bigr)$
  be the Grassmannian of linear subspaces in $L$ with codimension $m$
  considered as a real manifold, and let $\pi_1$, $\pi_2$ be the canonical
  projection maps from the universal family in $L \times G$ onto the factors.
  This universal family is simply the subvariety of the product whose fibre
  over a given point in $G$ is the corresponding codimension-$m$ linear
  subspace itself.  The previous paragraph shows that the map $\psi : W \to G$
  sending $[v]$ to the linear subspace $L \cap T_v$ is well-defined and
  continuous.  Since $\pi_1$ is a continuous map and $\pi_2$ is an open map,
  we see that
  $U_{\varepsilon} = \psi^{-1} \bigl( \pi_2 \bigl( \pi_1^{-1}(L \cap
  B_{\varepsilon}(g) \bigr) \bigr)$ is an open subset in $W$ considered as a
  real manifold.

  Finally, if the image of $\varphi$ does not contain $U_\varepsilon$ for all
  sufficiently small $\varepsilon > 0$, then there is a sequence of increasing
  positive integers $r_i$, a sequence of nonzero vectors $v_i \in V$, and a
  sequence of quadratic forms $f_i \in S_2 \setminus S_2^+$ such that the
  $[v_i] \in \PP(V)$ converge to $\varphi(g)$ as $i \to \infty$ and
  $f_i \in L \cap T_{v_i} \cap B_{1/r_i}(g)$ for each $i \in \NN$.  Thus, the
  quadratic forms $f_i$ converge to $g$ as $i \to \infty$.  However, for
  sufficiently large $i \in \NN$, the symmetric matrix corresponding to $f_i$
  has a negative eigenvalue and corank $1$ because
  $f_i \in S_2 \setminus S_2^+$ and $f_i \in B_{1/r_i}(g)$. Hence, the limit
  of the $f_i$ cannot be both positive-semidefinite and have corank $1$.  We
  conclude that, for sufficiently small $\varepsilon$, the elements in
  $L \cap T_v \cap B_{\varepsilon}(g)$ for all $[v] \in U_\varepsilon$ are
  positive-semidefinite.  Therefore, the image of $\varphi$ contains
  $U_\varepsilon$ for sufficiently small $\varepsilon$.
\end{proof}

Building on Lemma~\ref{lem:wellDefined}, we introduce the following class of
spectrahedral cones.  This definition guarantees that, both globally and
locally, the spectrahedral cone has numerous facets.

\begin{definition}
  \label{def:many-faceted}
  A spectrahedral cone $C = L \cap S_2^+$ is \define{many-faceted} if the
  points with corank $1$ form a dense subset of $\partial C$ and, for all
  $g \in \partial C$ with corank $1$ and all sufficiently small
  $\varepsilon > 0$, the image of
  $\varphi \colon B_{\varepsilon}(g) \cap \partial C \to \PP(V)$ contains a
  neighbourhood of $\varphi(g)$.
\end{definition}

\noindent
Being many-faceted is an extrinsic property; it depends on the presentation of
the spectrahedral cone.  For instance, if $C = L \cap S_2^+$ is many-faceted,
then we must have $\dim(L) \geq \dim(V) + 1$.

Two modest examples help illuminate this definition.  

\begin{example}[A spectrahedral cone that is not many-faceted]
  \label{exa:flatFace}
  Let $V \coloneq \RR^3$ and $S_2 \coloneq \RR[x_0,x_1,x_2]_2^{}$.  For the
  spectrahedral cone $C \coloneq L \cap S_2^+$ given by the linear subspace
  $L \coloneq \Span(x_0^2, x_0^{} x_2^{}, x_1^2, x_2^2) \subset S_2$, we have
  $C = \{ \alpha x_0^2 + 2 \beta x_0^{} x_2^{} + \gamma x_1^2 + \delta x_2^2 :
  \text{$\alpha \geq 0$, $\gamma \geq 0$, and
    $\alpha \delta - \beta^2 \geq 0$} \}$ and the associated symmetric
  matrices have the form 
  \[
    \begin{bmatrix}
    \alpha & 0 & \beta \\
    0 & \gamma & 0 \\
    \beta & 0 & \delta \\
    \end{bmatrix} \, .
  \]
  The relative interior of the face given by $\gamma = 0$ is open in the
  boundary $\partial C$ and consists of points with corank $1$ because the
  kernel of each quadratic form in the relative interior of this face is equal
  to
  $\Span \bigl( [ \begin{smallmatrix} 0 & 1 & 0 \end{smallmatrix}
  ]^{\transpose} \bigr)$.  However, the image of the map $\varphi$ is a single
  point in $\PP(V)$, so it does not contain an open subset. Thus, this
  spectrahedral cone is not many-faceted.  \hfill $\diamond$
\end{example}

\begin{example}[A spectrahedral cone that is many-faceted]
  As in Example~\ref{exa:flatFace}, let $V \coloneq \RR^3$ and
  $S_2 \coloneq \RR[x_0,x_1,x_2]_2^{}$.  For $C \coloneq L \cap S_2^+$ defined
  by
  $L \coloneq \Span(x_0^2 + x_1^2 + x_2^2, x_0^{} x_1^{}, x_0^{} x_2^{},
  x_1^{} x_2^{}) \subset S_2$, it follows that
  \[
    C = \left\{ \alpha x_0^2 + 2 \beta x_0^{} x_1^{} + 2 \gamma x_0^{} x_2^{}
      + \alpha x_1^2 + 2 \delta x_1^{} x_2^{} + \alpha x_2^2 :
    \begin{array}{c}
      \text{$\alpha \geq 0$, $\alpha^2-\beta^2 \geq 0$,} \\ 
      % \text{$\alpha^2 - \delta^2 \geq 0$,
      % $\alpha^2 - \gamma^2 \geq 0$, and} \\
     \text{$\alpha^3 - \alpha \beta^2 - \alpha \gamma^2  -
      \alpha \delta^2 + 2 \beta \gamma \delta \geq 0$}
    \end{array}
  \right\}
  \]
  and the associated symmetric matrices have the form
  \[
    \begin{bmatrix}
      \alpha & \beta & \gamma \\
      \beta & \alpha & \delta \\
      \gamma & \delta & \alpha \\
    \end{bmatrix} \, . 
  \]  
  The algebraic boundary of the section of this cone determined by setting
  $\alpha = 1$ equals the Cayley cubic surface defined by the affine equation
  $1 - \beta^2 - \gamma^2 - \delta^2 +2 \beta \gamma \delta$; see
  Subsection~5.2.2 in \cite{BPT}.  From the well-known image of the boundary
  surface (see Figure~\ref{fig:samosa}), which is affectionately referred to
  as `The Samosa', we observe that the cone is many-faceted.
  \begin{figure}[!ht]
    \includegraphics[width=2.0cm]{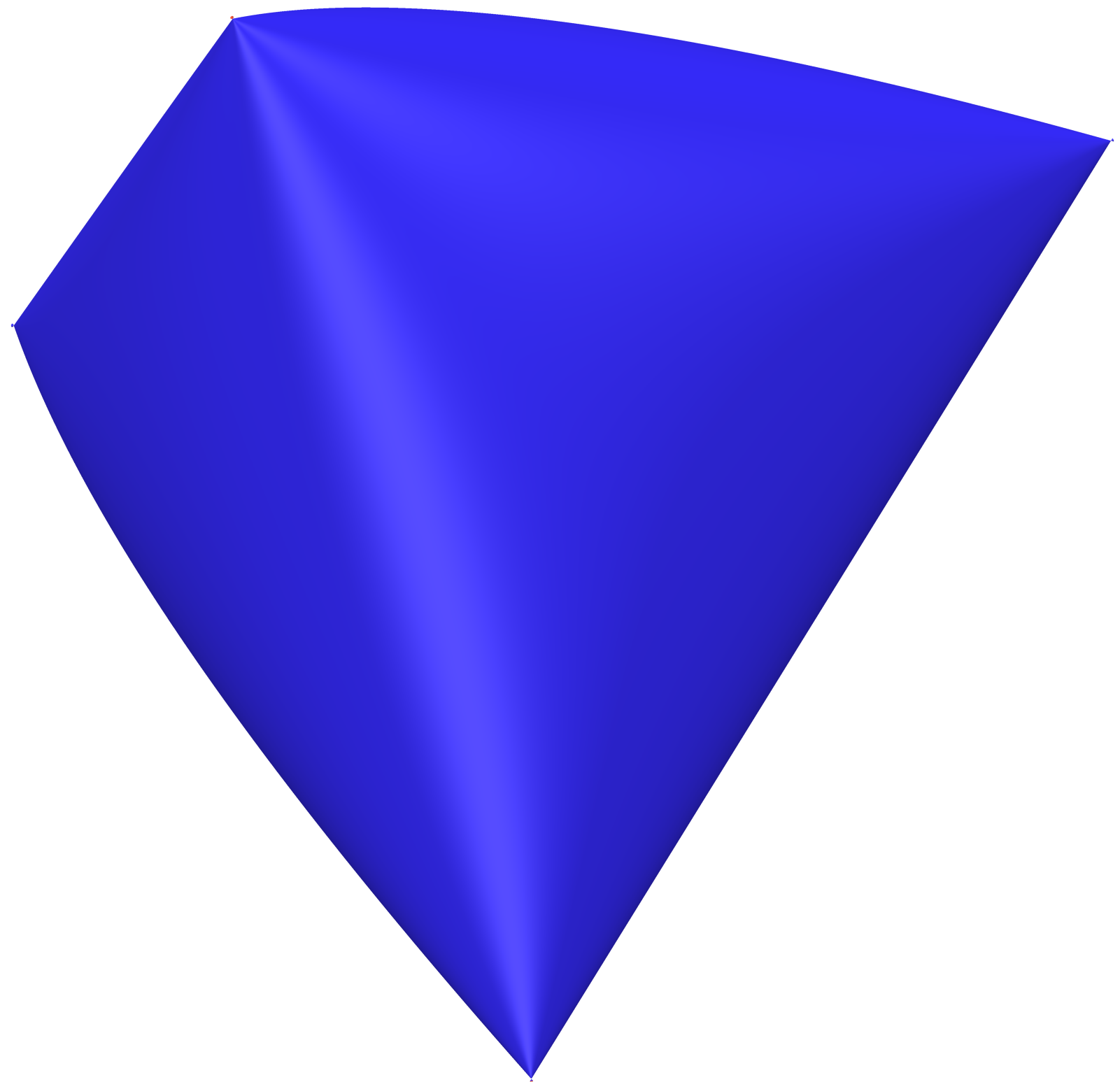}
    \hspace*{3cm}
    \includegraphics[width=2.0cm]{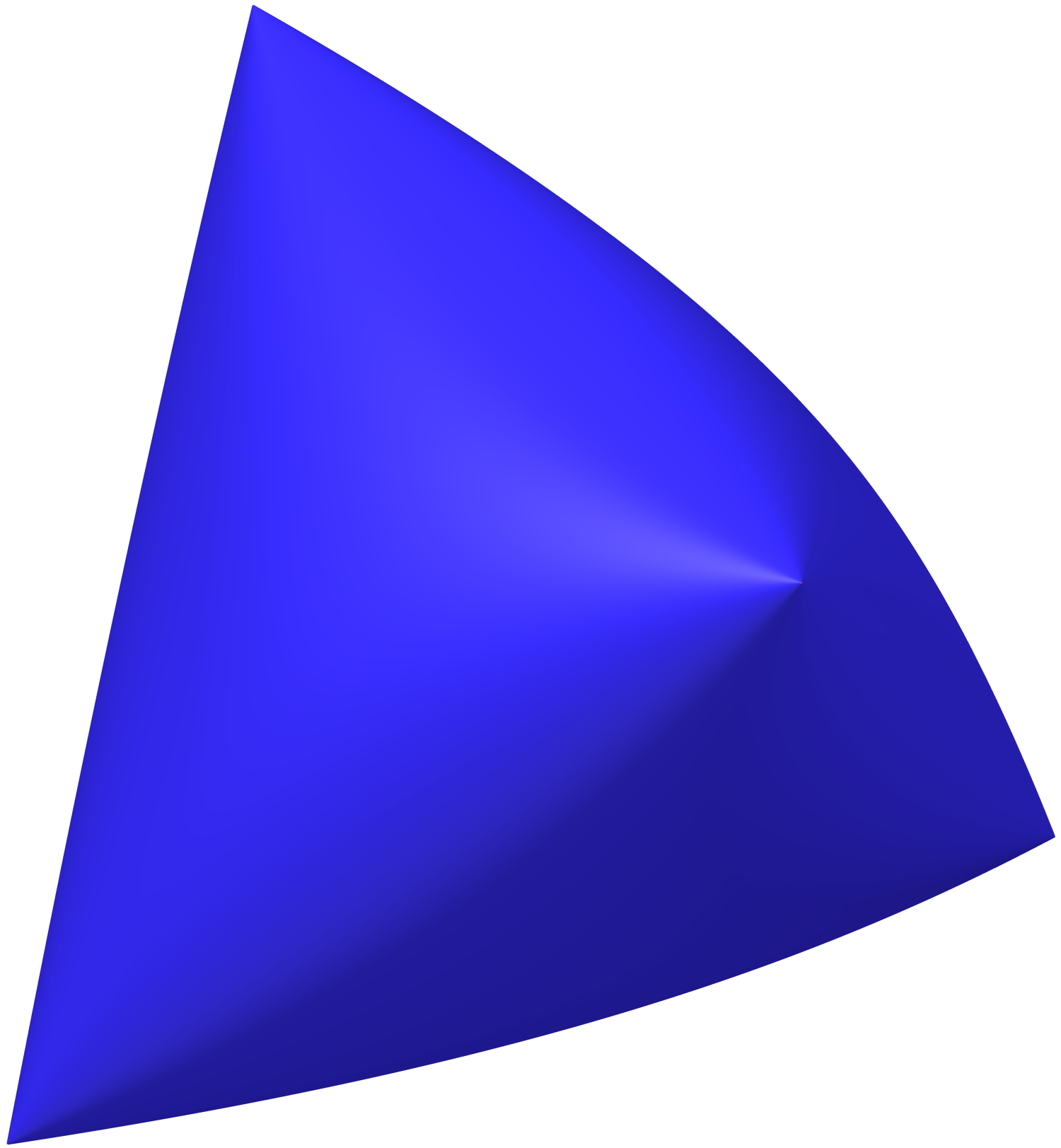}
    \caption{Two perspectives of `The Samosa'}
    \label{fig:samosa}
  \end{figure}
  For the quadratic form
  $g \coloneq x_0^2 + (x_1^{} - x_2^{})^2 \in \partial C$, we have
  $\varphi(g) = [ 0 : 1 : 1 ] \in \PP(V)$ and
  \[
    \dim(L \cap T_{\varphi(g)}) = \dim \Span \bigl( x_0^2 + (x_1^{} -
    x_2^{})^2, x_0^{} x_1^{} - x_0^{} x_2^{} \bigr) = 2 > 1 = \dim(L) - \dim(V)
    \, .
  \]
  This shows that there exist many-faceted spectrahedral cones which do not
  arise via Lemma~\ref{lem:wellDefined} . \hfill $\diamond$
\end{example}

% It is not possible for a $4$-dimensional linear subspace of
% $S_2 = \RR[x_0,x_1,x_2]_2$ to intersect the linear subspace $T_v$ transversely
% for all $v \in V$.  In contrast, a general $5$-dimensional subspace of $S_2$
% does.

% \begin{example}[More spectrahedral cones that are many-faceted]
%   As in Example~\ref{exa:flatFace}, let $V = \RR^3$ and
%   $S_2 = \RR[x_0,x_1,x_2]_2$.  If $L$ is the orthogonal subspace of an
%   indefinite irreducible quadratic form $f \in S_2$, then the cone
%   $C = L \cap S_2^+$ is the dual cone sum-of-squares cone, that is
%   $C = \Sos_{X,2}^*$ where $X = \variety(f) \subset \PP^2$.  \hfill $\diamond$
% \end{example}

To realize such many-faceted cones within convex algebraic geometry, consider
a real projective subscheme $X \subseteq \PP^n = \Proj(S)$ where
$S \coloneq \RR[x_0,x_1,\dotsc,x_n]$.  If $I_X$ is the saturated homogeneous ideal
defining $X$, then the $\ZZ$-graded coordinate ring of $X$ is $R \coloneq S/I_X$.
For each $j \in \ZZ$, the graded component $R_j$ of degree $j$ is a
finite-dimensional real vector space, and we set
\[
\Sos_{X,2j} \coloneq \bigl\{ f \in R_{2j} : \text{there exist
  $g_0, g_1, \dotsc, g_s \in R_{j}$ such that
  $f = g_0^2 + g_1^2 + \dotsb + g_s^2$} \bigr\} \, .
\]
Since a nonnegative real number has a square root in $\RR$, we see that
$\Sos_{X,2j}$ is a convex cone in $R_{2j}$.  The map
$\sigma_j \colon \Sym^2(R_j) \to R_{2j}$, induced by multiplication, is
surjective.  It follows that the cone $\Sos_{X,2j}$ is also full-dimensional
because the second Veronese embedding of $\PP^n$ is nondegenerate.  Moreover,
the dual map $\sigma_j^* \colon R_{2j}^* \to \Sym^2(R_j^*)$ is injective and,
for all $\ell \in R_{2j}^*$, the symmetric form
$\sigma_j^*(\ell) \colon R_j \otimes_\RR R_j \to \RR$ is given explicitly by
$g_1 \otimes g_2 \mapsto \ell(g_1g_2)$.

The subsequent proposition consolidates a few fundamental properties of this
cone and proves that many-faceted spectrahedral cones are common in convex
algebraic geometry.  A cone in a real vector space is \define{pointed} if it
is both closed in the Euclidean topology and contains no lines.

\begin{proposition}
  \label{pro:pointed}
  Fix $j \in \NN$.  If $X \subseteq \PP^n$ is a real projective subscheme with
  $\ZZ$-graded coordinate ring $R$ such that the map
  $\eta_g \colon R_j \to R_{2j}$ defined by $\eta_g(f) = fg$ is injective for
  all nonzero $g \in R_j$, then the following are equivalent.
  \begin{enumerate}[\upshape (a)]
  \item The cone $\Sos_{X,2j}$ is pointed.
  \item No nontrivial sum of squares of forms of degree $j$ equals zero.
  \item The points of corank $1$ form a dense subset of
    $\partial \Sos_{X,2j}^*$ in the Euclidean topology.
  \item The dual $\Sos_{X,2j}^*$ is a many-faceted spectrahedral cone.
  \end{enumerate}
\end{proposition}

\begin{proof}
  (a) $\Rightarrow$ (b): If some nontrivial sum of squares equals zero, then
  there exist $g_0, g_1, \dotsc, g_s \in R_j$ satisfying
  $g_0^2 + g_1^2 + \dotsb + g_s^2 = 0$.  We have $s > 0$ because the map
  $\eta_g$ is injective for all nonzero $g \in R_j$.  Since
  $g_0^2 = - (g_1^2+ g_2^2 + \dotsb + g_s^2)$, it follows that
  $\lambda g_0^2 \in \Sos_{X,2j}$ for all $\lambda \in \RR$ which contradicts
  the assumption that $\Sos_{X,2j}$ contains no lines.

  (b) $\Rightarrow$ (a): Fix an inner product on the real vector space $R_j$
  and let $g \mapsto \norm{g}$ denote the associated norm.  The spherical
  section $K \coloneq \{ g^2 \in R_{2j} : \text{$g \in R_j$ satisfies
    $\norm{g} = 1$} \}$ is compact because it is the continuous image of a
  compact set.  Moreover, the convex hull of $K$ does not contain $0$ because
  no nontrivial sum of squares equals zero.  Since $\Sos_{X,2j}$ is the
  conical hull of $K$, the cone $\Sos_{X,2j}$ is closed.  If $\Sos_{X,2j}$
  contains a line, then there exists a nonzero $f \in R_{2j}$ such that both
  $f$ and $-f$ lie in $\Sos_{X,2j}$.  However, it follows that the nontrivial
  sum $f + (-f)$ equals zero, which contradicts (b).

  (a) $\Rightarrow$ (c): Since $\Sos_{X,2j}$ is a pointed full-dimensional
  cone, its dual $\Sos_{X,2j}^*$ is also a pointed full-dimensional cone.  As
  a consequence, Theorem~2.2.4 in \cite{Schneider} implies that the linear
  functionals $\ell\in \Sigma_{X,2j}^*$ whose normal cone is a single ray form
  a dense subset of $\partial \Sigma_{X,2j}^*$.  We claim that every such
  linear functional $\ell$ has corank one.  If $f, g \in R_j$ are two nonzero
  elements lying in the kernel of $\sigma_j^*(\ell)$, then $f^2$ and $g^2$ are
  nonzero elements of the normal cone of $\Sigma_{X,2j}^*$ at $\ell$.  Because
  this normal cone is a ray, there exists a positive $\lambda \in \RR$ such
  that $f^2 = \lambda g^2$.  Hence, we have
  $(f + \smash{\sqrt{\lambda}} g)(f - \smash{\sqrt{\lambda}} g) = 0$ in
  $R_{2j}$. By injectivity of multiplication maps, we conclude that $f$ and
  $g$ are linearly dependent, so $\ell$ has corank $1$ and (c) holds.

  (c) $\Rightarrow$ (a): If $\Sos_{X,2j}$ is not closed, then the `(b)
  $\Rightarrow$ (a)' step shows that there is a nontrivial sum of squares from
  $R_j$ equal to zero in $R_{2j}$.  In this case, the `(a) $\Rightarrow$ (b)'
  step demonstrates that $\Sos_{X,2j}$ contains a line.  Now, if $\Sos_{X,2j}$
  contains a line, then its dual $\Sos_{X,2j}^*$ is not full-dimensional.
  Since the dual map $\sigma_j^* \colon R_{2j}^* \to \Sym^2(R_j^*)$ is
  injective, the linear subspace $\sigma_j^*(R_{2j}^*)$ does not intersect the
  interior of the cone $S_2^+$ consisting of positive-semidefinite forms in
  $\Sym^2(R_j^*)$.  Hence, the image $\sigma_j^*(\Sos_{X,2j}^*)$ consists of
  symmetric forms of corank at least $1$.  It follows that the boundary
  consists of symmetric forms of corank at least $2$, which contradicts (c).

  (a) $\Leftrightarrow$ (d): Let $V \coloneq R_j$ and let
  $S_2 \coloneq \Sym^2(R_j^*)$.  For a linear functional
  $\ell \in \Sos_{X,2j}^*$, we have $\ell(g^2) \geq 0$ for all $g \in R_j$, so
  the symmetric form $\sigma_j^*(\ell)$ is positive-semidefinite.  Conversely,
  if $\sigma_j^*(\ell)$ is positive-semidefinite symmetric form, then we have
  $\ell(g^2) \geq 0$ for all $g \in R_j$.  It follows that
  $\ell( g_0^2 + g_1^2 + \dotsb + g_s^2) = \ell(g_0^2) + \ell(g_1^2) + \dotsb
  + \ell(g_s)^2 \geq 0$ for $g_0, g_1, \dotsc, g_s \in R_j$ and
  $\ell \in \Sos_{X,2j}^*$.  Hence, the map $\sigma_j^*$ identifies the dual
  $\Sos_{X,2j}^*$ with the spectrahedral cone determined by the linear
  subspace $L \coloneq \sigma_j^*(R_{2j}^*)$ in $S_2 = \Sym^2(R_{j}^*)$;
  compare with Lemma~2.1 in \cite{BSV}. Given a nonzero $f \in V$, let
  $T_f \subset S_2$ be the linear subspace consisting of the symmetric forms
  $h \in S_2$ such that $f \in \Ker(h)$.  As in the proof of
  Lemma~\ref{lem:wellDefined}, we have $\codim T_f = \dim V$.  The map
  $\sigma_j^*$ identifies the linear subspace $L \cap T_f$ with the set of
  linear functionals $\ell \in R_{2j}^*$ such that $\ell(fg) = 0$ for all
  $g \in V$.  If $\langle f \rangle$ denotes the ideal in $R$ generated by
  $f$, then it follows that the codimension of $L \cap T_f$ in $L$ equals the
  dimension of $\langle f \rangle_{2j}$.  By hypothesis, the map
  $\eta_f \colon R_j \to R_{2j}$ is injective, so we obtain
  $\dim \langle f \rangle_{2j} = \dim R_{j} = \dim V$.  Hence, we have
  $\dim L + \dim T_f - \dim L \cap T_f = \dim S_2$ and the linear subspaces
  $L$ and $T_f$ meet transversely for all nonzero $f \in V$.  If
  $g \in \partial \Sigma_{X,2j}^*$ has corank $1$ and $\varepsilon > 0$ is
  sufficiently small, then Lemma~\ref{lem:wellDefined} establishes that the
  image of
  $\varphi \colon B_{\varepsilon}(g) \cap \partial \Sos_{X,2j}^* \to \PP(V)$
  contains a neighbourhood of $\varphi(g)$.  Since `(a) $\Leftrightarrow$ (c)'
  establishes that $\Sos_{X,2j}$ is pointed if and only if the points of
  corank $1$ form a dense subset of $\partial \Sos_{X,2j}^*$ in the Euclidean
  topology, we conclude that $\Sigma_{X,2j}$ is pointed if and only if its
  dual $\Sos_{X,2j}^*$ is a many-faceted spectrahedral cone.
\end{proof}

\begin{remark}
  The first condition in Proposition~\ref{pro:pointed} may be rephrased. 
  %using different terminology.  
  A cone $C$ is salient if it does not contain an opposite pair of nonzero
  vectors, that is $(-C) \cap C \subseteq \{0\}$.  In other words, a cone is
  salient if and only if it contains no lines.  Hence, a cone is pointed if it
  is both closed and salient.
\end{remark}

\begin{remark}
  \label{rem:nonclosed}
  If $\Sos_{X,2j}$ is not closed, then the `(c) $\Rightarrow$ (a)' step proves
  that $\Sos_{X,2j}$ contains a line.
\end{remark}

We end this section with special cases of Proposition~\ref{pro:pointed}.  A
subscheme $X \subseteq \PP^n$ is a real projective variety if it is a
geometrically integral projective scheme over $\RR$.  Moreover, a real variety
$X$ is totally real if the set $X(\RR)$ of real points is Zariski dense.  The
most important application of Proposition~\ref{pro:pointed} is the following
corollary.

\begin{corollary}
  \label{cor:variety}
  Let $X \subseteq \PP^n$ be a real projective variety. The cone $\Sos_{X,2j}$
  is pointed if and only if its dual $\Sos_{X,2j}^*$ is a many-faceted
  spectrahedral cone.  Furthermore, the cones $\Sos_{X,2j}^*$ are many-faceted
  for all $j \in \NN$ if and only if $X$ is totally real.
\end{corollary}

\begin{proof}
  Because $X$ is geometrically integral, its coordinate ring $R$ is a domain.
  Hence, each nonzero element in $R_j$ is a nonzerodivisor and the map
  $\eta_g \colon R_j \to R_{2j}$ is injective for all nonzero $g \in R_j$.  By
  combining this observation with Proposition~\ref{pro:pointed}, we first
  conclude that $\Sos_{X,2j}$ is pointed if and only if its dual
  $\Sos_{X,2j}^*$ is a many-faceted spectrahedral cone.  Secondly, $X$ is
  totally real if and only if, for all $j \in \NN$, no nontrivial sum of
  squares from $R_j$ equals zero in $R_{2j}$; compare with Lemma~2.1 in
  \cite{BSV}.  Therefore, the first part together with
  Proposition~\ref{pro:pointed} establishes that the cones $\Sos_{X,2j}^*$ are
  many-faceted for every $j \in \NN$ if and only if $X$ is totally real.
\end{proof}

%%%%%%%%%%%%%%%%%%%%%%%%%%%%%%%%%%%%%%%%%%%%%%%%%%%%%%%%%%%%%%%%%%%%%%%%%%%%%%
\section{A Bertini Theorem for Separators}
\label{sec:Bertini}

\noindent
In this section, we explore the properties of separating hyperplanes within
convex algebraic geometry.  Two cones $C_1$ and $C_2$ in a real vector space
are \define{well-separated} if there exists a linear functional $\ell$ such
that $\ell(v) > 0$ for all nonzero $v \in C_1$ and $\ell(v) < 0$ for all
nonzero $v \in C_2$.  A linear functional $\ell$ with these properties is
called a \define{strict separator}.  If $C_1$ and $C_2$ are pointed (closed
and contain no lines), then being well-separated is equivalent to
$C_1 \cap C_2 = \{ 0 \}$.

The main result in this section is an analogue of Bertini's Theorem to convex
algebraic geometry.  As in Section~\ref{sec:many-faceted}, $X \subseteq \PP^n$
is a real projective subscheme with $\ZZ$-graded coordinate ring $R = S/I_X$
and $S = \RR[x_0, x_1, \dotsc, x_n]$.  Given an element $f \in R_{2j}$, we set
$f \cdot \Sos_{X,2k} \coloneq \{ fg \in R_{2j+2k} : g \in \Sos_{X,2k} \}$.
For a nonzero homogeneous polynomial $h \in S$, the associated hypersurface
section of $X$ is the subscheme
$X' \coloneq X \cap \variety(h) \subset \PP^n$.  The $\ZZ$-graded coordinate
ring of $X'$ is the quotient $R' \coloneq S/I_{X'}$ where $I_{X'}$ is the
saturated homogeneous ideal
$(I_X + \langle h \rangle : \langle x_0, x_1, \dotsc, x_n\rangle^\infty)$.  We
write $f' \in R'_{2j}$ for the canonical image of $f \in R_{2j}$.

\begin{theorem}
  \label{thm:Bertini}
  Fix positive integers $j$ and $k$. Let $X \subseteq \PP^n$ be a real
  projective subscheme with coordinate ring $R$ such that the map
  $\eta_g \colon R_{j+k} \to R_{2j+2k}$ is injective for all nonzero
  $g \in R_{j+k}$, and consider a nonzerodivisor $f \in R_{2j}$.  If the cones
  $\Sos_{X,2j+2k}$ and $f \cdot \Sos_{X,2k}$ are well-separated, then the set
  of hypersurface sections $X'$ of $X$, such that $\Sos_{X',2j+2k}$ and
  $f' \cdot \Sos_{X',2k}$ are well-separated, contains a nonempty open subset
  of $\PP(R_{j+k})$ in the Euclidean topology.
\end{theorem}

\begin{proof}
  To begin, we prove that the cones $\Sos_{X,2j+2k}$ and
  $f \cdot \Sos_{X,2k}$ are pointed.  By hypothesis, $\Sos_{X,2j+2k}$ and
  $f \cdot \Sos_{X,2k}$ are well-separated, so neither cone contains a line.
  Hence, Remark~\ref{rem:nonclosed} shows that $\Sos_{X,2j+2k}$ is also
  closed.  As $f$ is a nonzerodivisor, the map
  $\eta_f \colon R_{2k} \to R_{2j+2k}$ is injective, so the cone $\Sos_{X,2k}$
  is isomorphic to the cone $f \cdot \Sos_{X,2k}$.  Hence, a second
  application of Remark~\ref{rem:nonclosed} establishes that both
  $\Sos_{X,2k}$ and $f \cdot \Sos_{X,2k}$ are closed.

  Now, let $C \coloneq \Sos_{X,2j+2k}^* \cap (-f \cdot \Sos_{X,2k})^*$ be the
  cone of separators.  The cone $C$ is closed and full-dimensional because
  $\Sos_{X,2j+2k}$ and $f \cdot \Sos_{X,2k}$ are well-separated.  In
  particular, the boundary of $C$ is not contained in the boundary of
  $\Sos_{X,2j+2k}^*$ or the boundary of $(-f \cdot \Sos_{X,2k})^*$.  Since the
  cone $(-f \cdot \Sos_{X,2k})^*$ is full-dimensional and
  \begin{align*}
    D &\coloneq \partial C \setminus \partial (-f \cdot \Sos_{X,2k})^* \\
    &= \Bigl( \bigl( \partial (-f \cdot \Sos_{X,2k})^* \cap \Sos_{X,2j+2k}^*
      \bigr) \cup \bigl( (-f \cdot \Sos_{X,2k})^* \cap \partial
      \Sos_{X,2j+2k}^* \bigr) \Bigr) \setminus \partial (-f \cdot
      \Sos_{X,2k})^* \\
    &= \bigl( (-f \cdot \Sos_{X,2k})^* \setminus \partial(-f \cdot
      \Sos_{X,2k})^* \bigr) \cap \partial \Sos_{X,2j+2k}^* \, ,
  \end{align*}
  it follows that $D$ is a nonempty open subset of $\partial \Sos_{X,2j+2k}^*$
  in the Euclidean topology.  Since $\Sos_{X,2j+2k}$ is pointed,
  Proposition~\ref{pro:pointed} implies that $\Sos_{X,2j+2k}^*$ is a
  many-faceted spectrahedral cone.  Hence, the points with corank $1$ form a
  dense subset of $\partial \Sos_{X,2j+2k}^*$, so we may choose $g \in D$ with
  corank $1$. Moreover, for a sufficiently small $\varepsilon > 0$, the image
  of the map $\varphi \colon B_{\varepsilon}(g) \cap D \to \PP(R_{j+k})$
  contains a neighbourhood $U$ of $\varphi(g)$.  Hence, if
  $\ell \in R_{2j+2k}^*$ satisfies $[\varphi(\ell)] \in U$, then there exists 
  $h \in R_{j+k}$ such that $\Ker \sigma_{j+k}^*(\ell) = \Span(h)$.  Let
  $X' \coloneq X \cap \variety(h)$ denote the corresponding hypersurface section
  with coordinate ring $R'$.  Since $\ell'$ has corank $1$, the linear
  functional $\ell \in R_{2j+2k}^*$ induces a strict separator
  $\ell' \in (R')_{2j+2k}^*$ on the cones $\Sos_{X',2j+2k}$ and
  $f' \cdot \Sos_{X',2k}$.  Therefore, the set of $X'$, such that
  $\Sos_{X',2j+2k}$ and $f' \cdot \Sos_{X',2k}$ are well-separated, contains
  the nonempty open subset $U$ of $\PP(R_{j+k})$.
\end{proof}

To exploit Theorem~\ref{thm:Bertini}, we also need to understand the
properties of strict separators on zero-dimensional schemes.  As we will see,
the existence of strict separators imposes nontrivial constrains on a set of
points. For a real projective scheme $X \subseteq \PP^n$ with homogeneous
coordinate ring $R$, the Hilbert function $\HF_X \colon \ZZ \to \ZZ$ is
defined by $\HF_X(j) \coloneq \dim_\RR R_j$.  Following Section~3.1 in
\cite{HarrisMontreal} or Section~2 in \cite{HarrisGenus}, a set of points
$X \subseteq \PP^n$, that is a zero-dimensional reduced subscheme, has the
\define{uniform position property} if the Hilbert function of a subset of $X$
depends only on the cardinality of the subset.

The concluding proposition of this section shows that the existence of certain
positive linear functionals on a set of points imposes constraints on its
Hilbert function.

\begin{proposition}
  \label{pro:zero}
  Fix positive integers $j$ and $k$, and let $X \subseteq \PP^n$ be a set of
  at least two points with the uniform position property.
  \begin{enumerate}[\upshape (i)]
  \item Suppose that $X$ has no real points.  If there exists a linear
    functional $\ell \in R_{2k}^*$ that is positive on the nonzero elements in
    $\Sos_{X,2k}$, then we have
    $\HF_X(k) \leq \bigl\lceil \frac{1}{2} \HF_X(2k) \bigr\rceil$.
  \item Suppose that $f \in R_{2j}$ is positive on $X(\RR)$ and does not
    vanish at any point in $X(\CC)$.  If there exists a linear functional
    $\ell \in R_{2j+2k}^*$ that is a strict separator for $\Sos_{X,2j+2k}$ and
    $f \cdot \Sos_{X,2k}$, then we have
    $\HF_X(k) + \HF_X(j+k) \leq \HF_X(2j+2k)$.
  \end{enumerate}
\end{proposition}

\begin{proof}
  We begin with an analysis of the symmetric forms arising from point
  evaluations.  Let $Z \subseteq X$ be a subset consisting of $e$ distinct
  real points and $m$ complex conjugate pairs.  Choose affine representatives
  $\tilde{p}_1, \tilde{p}_2, \dotsc, \tilde{p}_e \in \AA^{n+1}(\RR)$ for
  points in $Z(\RR)$, and choose affine representatives
  $\tilde{a}_1 \pm \tilde{b}_1 \sqrt{-1}, \tilde{a}_2 \pm \tilde{b}_2
  \sqrt{-1}, \dotsc, \tilde{a}_m \pm \tilde{b}_m \sqrt{-1} \in \AA^{n+1}(\RR)$
  where $\tilde{a}_j, \tilde{b}_j \in \AA^{n+1}(\RR)$ for the complex
  conjugate pairs in $Z(\CC)$.  For any $p \in X(\RR)$ and any $k \in \NN$,
  evaluation at an affine representative $\tilde{p} \in \AA^{n+1}(\RR)$
  determines the linear functional $\tilde{p}^* \in R_{2k}^*$.  Any linear
  functional $\ell \in R_{2k}^*$ lying in the span of these point evaluations
  can be written as
  \begin{align*}
    \ell 
    &= \sum_{i=1}^e \kappa_i \tilde{p}_i^* + \sum_{j=1}^m \bigl( (\lambda_j
      + \mu_j \sqrt{-1})(\tilde{a}_j + \tilde{b}_j \sqrt{-1})^* + (\lambda_j
      - \mu_j \sqrt{-1})(\tilde{a}_j - \tilde{b}_j \sqrt{-1})^* \bigr) 
    % &= \sum_{i=1}^e \kappa_i \tilde{p}_i^* + \sum_{j=1}^m 2 \lambda_j
    %   \tilde{a}_j^* - 2 \mu_{j} \tilde{b}_j^*
  \end{align*}
  where $\kappa_i, \lambda_j, \mu_j \in \RR$ for $1 \leq i \leq e$ and
  $1 \leq j \leq m$.  It follows that
  \[
  \sigma_{k}^*(\ell) = \sum_{i=1}^e \kappa_i (\tilde{p}_i^*)^2 + \sum_{j=1}^m
  \lambda_j \bigl( (\tilde{a}_j^*)^2 - (\tilde{b}_j^*)^2 \bigr) - 2 \mu_j
  (\tilde{a}_j^*) (\tilde{b}_j^*) \in \Sym^2(R_k^*) \, .
  \]
  The eigenvalues for the symmetric matrix 
  \[
    \begin{bmatrix}
      \phantom{-}\lambda_j & - \mu_j \\
      -\mu_j & -\lambda_j \\
    \end{bmatrix}
  \]
  are $\pm \sqrt{\smash[b]{\lambda_j^2 + \mu_j^2}}$, so the number of positive
  eigenvalues for $\sigma_{k}^*(\ell)$ is at most the number $e_+$ of
  positive $\kappa_i$ plus the number $m'$ of nonzero $\lambda_j^2 + \mu_j^2$.
  Similarly, the number of negative eigenvalues for $\sigma_{k}^*(\ell)$ is
  at most the number $e_{-}^{}$ of negative $\kappa_i$ plus the number $m'$ of
  nonzero $\lambda_j^2 + \mu_j^2$.  Hence, if $\sigma_{k}^*(\ell)$ is
  positive-definite, then we have $\HF_X(k) = \dim R_k^* \leq e_{+} + m'$.

  Using this analysis, we prove (i).  Assume that $\ell \in R_{2k}^*$ is
  positive on the nonzero elements in $\Sos_{X,2k}$.  A form in $R_{2k}^*$ is
  zero if and only if it is annihilated by $\tilde{p}^* \in R_{2k}^*$ for all
  points $p \in X(\CC)$.  Hence, every linear functional in $R_{2k}^*$ can be
  written as a $\CC$\nobreakdash-linear combinations of such point
  evaluations.  The evaluations at the points in any subset $X$, with
  cardinality at least $\HF_X(2k)$, span $R_{2k}^*$ because $X$ has the
  uniform position property.  As $X$ is a set of points, the value of Hilbert
  function $\HF_X(k)$ is at most the number of points.  Since
  $X(\RR) = \varnothing$, we may choose $m$ conjugate pairs of points in
  $X(\CC)$ with $m \coloneq \bigl\lceil \tfrac{1}{2} \HF_X(2k) \bigr\rceil$;
  in particular, we have $e = 0$.  Since $\sigma_{k}^*(\ell)$ is
  positive-definite, the first paragraph shows that
  $\HF_X(k) \leq m' \leq m = \bigl\lceil \tfrac{1}{2} \HF_X(2k) \bigr\rceil$
  as required.

  We next examine the symmetric forms induced by the element $f \in R_{2k}$.
  Given $\ell \in R_{2k+2k}^*$, the linear functional $\ell' \in R_{2k}^*$ is
  defined by $\ell'(g) \coloneq \ell(fg)$ for all $g \in R_{2k}$.  When
  $\ell \in R_{2j+2k}^*$ lies in the span of the point evaluations for $Z$,
  the expression for $\ell'$ as a linear combination of the point evaluations
  has the same number of positive, negative, and nonzero coefficients as
  $\ell$ because $f \in R_{2j}$ is positive on $X(\RR)$ and does not vanish at
  any points in $X(\CC)$.  Hence, if $\sigma_{k}^*(\ell')$ is
  negative-definite, then the first paragraph implies that
  $\HF_X(k) = \dim R_k^* \leq e_{-}^{} + m'$.

  Lastly, we establish (ii).  Assume that $\ell \in R_{2j+2k}^*$ is a strict
  separator for $\Sos_{X,2j+2k}$ and $f \cdot \Sos_{X,2k}$.  As in the second
  paragraph, we may choose a subset of $X$ such that the point evaluations
  span $R_{2j+2k}^*$.  Suppose that there exists a conjugate-invariant basis
  of $R_{2j+2k}^*$ consisting of point evaluations at $e$ distinct real points
  and $m$ complex conjugate pairs.  Since $\sigma_{k}^*(\ell')$ is
  negative-definite and $\sigma_{j+k}^*(\ell)$ is positive-definite, the
  first and third paragraphs combine to show that
  \[
  \HF_X(k) + \HF_X(j+k) \leq (e_{-}^{} +m') + (e_{+} + m') \leq e +
  2m = \HF_X(2j+2k) \, .
  \]
  On the other hand, if no subset of $X$ yields a conjugate-invariant basis of
  $R_{2j+2k}^*$, then there are
  $m \coloneq \bigl\lceil \tfrac{1}{2} \HF_X(2j + 2k) \bigr\rceil$ conjugate
  pairs of points in $X(\CC)$ that span $R_{2j+2k}^*$ and we have $e = 0$.
  Hence, we obtain $\HF_X(j+k) \leq m$,
  $2 \HF_X(j+k) \leq 2m = \HF_X(2j+2k) +1$, and
  $2 \HF_X(j+k) -1 \leq \HF_X(2j+2k)$.  With the goal of finding a
  contradiction, assume that $\HF_X(k) + \HF_X(j+k) > \HF_X(2j+2k)$.  It
  follows that $\HF_X(k) + 1 > \HF_X(j+k)$.  The Hilbert function of a set of
  points is strictly increasing until it stabilizes at the number points, so
  we deduce that $\HF_X(k) = \HF_X(j+k) = \HF_X(2j+2k)$.  Hence, the
  inequality $2 \HF_X(j+k) \leq \HF_X(2j+2k) +1$ implies that
  $\HF_X(j+k) = 1$.  However, this contradicts the hypothesis that $X$ has at
  least two points.  Therefore, we conclude that
  $\HF_X(k) + \HF_X(j+k) \leq \HF_X(2j+2k)$.
\end{proof}

%%%%%%%%%%%%%%%%%%%%%%%%%%%%%%%%%%%%%%%%%%%%%%%%%%%%%%%%%%%%%%%%%%%%%%%%%%%%%%
\section{Upper Bounds for Sum-of-Squares Multipliers}
\label{sec:boundingMultipliers}

\noindent
This section establishes an upper bound on the minimal degree of a
sum-of-squares multiplier.  These geometric degree bounds for the existence of
multipliers prove the first halves of our main theorems.  After a preparatory
lemma, Theorem~\ref{thm:upperBound} describes the general result for curves
and is followed by several corollaries and valuable examples.  The same
approach is then applied to higher-dimensional varieties to obtain the general
Theorem~\ref{thm:totallyreal}.  The ensuing examples illustrate the
applicability of this theorem.

Throughout this section, we work with a real projective subscheme
$X \subseteq \PP^n$ with $\ZZ$-graded coordinate ring $R = S/I_X$.  The sign
of an element $f \in R_{2j}$ at a real point $p \in X(\RR)$ is defined to be
$\sign_{p}(f) \coloneq \sign\bigl( \tilde{f}(\tilde{p}) \bigr) \in \{-1,0,1\}$,
where the polynomial $\tilde{f} \in S_{2j}$ maps to $f$ and the nonzero real
point $\tilde{p} \in \AA^{n+1}(\RR)$ maps to $p$ under the canonical quotient
maps.  Since $p \in X(\RR)$, the real number $\tilde{f}(\tilde{p})$ is
independent of the choice of $\tilde{f}$.  Similarly, the choice of affine
representative $\tilde{p}$ is determined up to a nonzero real number and the
degree of $f$ is even, so the value of $\tilde{f}(\tilde{p})$ is determined up
to the square of a nonzero real number.  Hence, the sign of $f \in R_{2j}$ at
$p \in X(\RR)$ is well-defined.  We simply write $f(p) \geq 0$ for
$\sign_p(f) \geq 0$.  The subset $\Pos_{X,2j} \coloneq \{ f \in R_{2j} :
\text{$f(p) \geq 0$ for all $p \in X(\RR)$} \}$ forms a pointed
full-dimensional convex cone in $R_{2j}$; see Lemma~2.1 in \cite{BSV}.

As our initial focus, a \define{curve} $X \subseteq \PP^n$ is a
one-dimensional projective variety.  Following Chapter~2 of \cite{Migliore},
the \define{deficiency module} (also known as the Hartshorne-Rao module) of
$X$ is the $\ZZ$-graded $S$\nobreakdash-module
$M_X \coloneq \bigoplus_{i \in \ZZ} H^1\bigl(\PP^n, \mathcal{I}_{X}(i) \bigr)$.  A
homogeneous polynomial $h \in S = \RR[x_0,x_1,\dotsc,x_n]$ determines the
$\ZZ$\nobreakdash-graded submodule
$\Ann_{M_X}(h) \coloneq ( 0 :_{M(X)} h ) = \{ f \in M_X : fh = 0 \}$ of the
deficiency module $M_X$.  The next lemma (cf.{} Proposition~2.1.2 in
\cite{Migliore}) shows that this submodule measures the failure of the ideal
$I_X + \langle h \rangle$ to be saturated.

\begin{lemma}
  \label{lem:radicality}
  Fix a positive integer $j$ and a nonnegative integer $k$. Let
  $X \subseteq \PP^n$ be a curve.  If $h \in S_k$ does not belong to the ideal
  $I_X$ and $X' \coloneq X \cap \variety(h)$ is the associated hypersurface section
  of $X$, then we have $\Ann_{M_X}(h)_{j-k} = 0$ if and only if
  $(I_{X'})_{j} = (I_X + \langle h \rangle)_{j}$.
\end{lemma}

\begin{proof}
  By definition, the submodule $\Ann_{M_X}(h)$ fits into the exact sequence
  \[
  0 \longrightarrow \bigl(\Ann_{M_X}(h) \bigr)(-k) \longrightarrow M_X(-k)
  \xrightarrow{\;\; \cdot h \;\;} M_X \, .
  \]
  Sheafifying the canonical short exact sequence
  $0 \longrightarrow I_X \cap \langle h \rangle \longrightarrow I_X \oplus
  \langle h \rangle \longrightarrow I_X + \langle h \rangle \longrightarrow 0$
  and taking cohomology of appropriate twists produces the long exact sequence
  \[
  0 \longrightarrow I_X(-k) \xrightarrow{\;\; \left[ 
      \begin{smallmatrix} 
        h \\ h 
      \end{smallmatrix} \right] \;\;} I_X \oplus \langle h \rangle
  \longrightarrow I_{X'} \longrightarrow M_X(-k) \xrightarrow{\;\; \cdot h
    \;\;} M_X \, .
  \]
  Breaking this long exact sequence into short exact sequences, we obtain
  \[
  0 \longrightarrow I_X + \langle h \rangle \longrightarrow I_{X'}
  \longrightarrow \bigl( \Ann_{M_X}(h) \bigr)(-k) \longrightarrow 0 \, .
  \]
  Thus, we have
  $\bigl( \Ann_{M_X}(h) \bigr)(-k) \cong I_{X'}/(I_X + \langle h \rangle)$ and
  the required equivalence follows.
\end{proof}

As a consequence of Lemma~\ref{lem:radicality}, we see that some natural
geometric conditions imply that the ideal $I_X + \langle h \rangle$ is
saturated.

\begin{remark}
  \label{rem:ACM}
  A curve $X$ is projectively normal if and only if $M_X = 0$.  With this
  hypothesis, Lemma~\ref{lem:radicality} implies that we have
  $(I_{X'})_{j} = (I_X + \langle h \rangle)_{j}$ for all $j \in \ZZ$.  In
  particular, if $X$ is arithmetically Cohen--Macaulay, then the ideal
  $I_X + \langle h \rangle$ is saturated.
\end{remark}

The next result is the general form of our degree bound for the existence of
sum-of-squares multipliers on curves.

\begin{theorem}
  \label{thm:upperBound}
  Fix a positive integer $j$ and a nonnegative integer $k$.  Let
  $X \subseteq \PP^n$ be a totally-real curve such that
  $H^1\bigl( \PP^n, \mathcal{I}_{X}(j+k) \bigr) = 0$ and
  $\HF_X(2j+2k) < 2 \HF_X(j+k) + \HF_X(k) - 1$.  For all $f \in \Pos_{X,2j}$,
  there exists a nonzero $g \in \Sos_{X,2k}$ such that
  $f g \in \Sos_{X,2j+2k}$.
\end{theorem}

\begin{proof}
  We start by reinterpreting the non-existence of a suitable multiplier
  $g \in \Sos_{X,2k}$ as the existence of a strict separator between
  appropriate cones.  Corollary~\ref{cor:variety} implies that the cones
  $\Sos_{X,2j+2k}$ and $\Sos_{X,2k}$ are pointed.  If $f = 0$, then the
  conclusion is trivial, so we may assume that $f$ is nonzero.  It follows
  that $f$ is a nonzerodivisor because $X$ is integral.  Since the map
  $\eta_f \colon R_{2k} \to R_{2j+2k}$ is injective, the pointed cone
  $\Sos_{X,2k}$ is isomorphic to the cone $f \cdot \Sos_{X,2k}$.  Hence, the
  non-existence of a nonzero $g \in \Sos_{X,2k}$ such that
  $fg \in \Sos_{X,2j+2k}$ is equivalent to saying that the cones
  $\Sos_{X,2j+2k}$ and $f \cdot \Sos_{X,2k}$ are well-separated.

  To complete the proof, we reduce to the case of points by using new and old
  Bertini Theorems.  Our convex variant, Theorem~\ref{thm:Bertini}, implies
  that the set of homogeneous polynomials $h \in S_{j+k}$, such that
  $h \not\in I_X$, $X' \coloneq X \cap \variety(h) \subset \PP^n$, and the
  cones $\Sos_{X',2j+2k}$ and $f' \cdot \Sos_{X',2k}$ are well-separated,
  contains a nonempty Euclidean open subset $U_1 \subseteq \PP(R_{j+k})$.  The
  classic version of Bertini's Theorem (see Th\'eor\`eme~6.3 in \cite{J})
  shows that there is a nonempty Zariski open subset
  $U_2 \subseteq \PP(R_{j+k})$ such that, for all $[h] \in U_2$, the
  hypersurface section $X'$ is a reduced set of points and $f$ does not vanish
  at any point in $X'$.  Moreover, our hypothesis that $(M_X)_{j+k} = 0$
  combined with Lemma~\ref{lem:radicality} establishes that there exists
  another nonempty Zariski open subset $U_3 \subseteq \PP(R_{j+k})$ such that,
  for all $[h] \in U_3$, we have
  $(I_{X'})_{2j+2k} = (I_X + \langle h \rangle)_{2j+2k}$, which implies that
  $\HF_{X'}(2j+2k) = \HF_X(2j+2k) - \HF_X(j+k)$.  The triple intersection
  $U_1 \cap U_2 \cap U_3$ is nonempty, so Proposition~\ref{pro:zero}~(ii)
  yields the inequality $\HF_{X'}(k) + \HF_{X'}(j+k) \leq \HF_{X'}(2j+2k)$.
  By construction, we have $\HF_{X'}(j+k) \leq \HF_X(j+k) - 1$ and
  $\HF_{X'}(i) \leq \HF_X(i)$ for all $i < j + k$.  Therefore, we conclude
  that $\HF_X(k) + 2 \HF_X(j+k) -1 \leq \HF_X(2j+2k)$ when the cones
  $\Sos_{X,2j+2k}$ and $f \cdot \Sos_{X,2k}$ are well-separated.
\end{proof}

The hypothesis in Theorem~\ref{thm:upperBound} may be recast using alternative
numerical invariants.  With this in mind, set
$\ee_i(X) \coloneq \max\{ i \in \ZZ : H^i\bigl( \PP^n, \mathcal{I}_X(i) \bigr) \neq
0 \}$,
so that the Castelnuovo--Mumford regularity of $\mathcal{I}_X$ equals
$\max\{ \ee_i(X)+i : i \in \ZZ\}$; compare with Theorem~4.3 in \cite{EisSyz}.

\begin{corollary}
  \label{cor:tripleMax}
  Fix a positive integer $j$ and a nonnegative integer $k$. Let
  $X \subseteq \PP^n$ be a totally-real curve of degree $d$ and arithmetic
  genus $\genus$, and assume that
  \[
  k \geq \max\bigl\{ \ee_1(X)+1, \tfrac{1}{2} \ee_2(X) + \tfrac{1}{2} - j,
  \tfrac{2 \genus}{d} \bigr\} \, .
  \]
  For all $f \in \Pos_{X,2j}$, there exists a nonzero $g \in \Sos_{X,2k}$ such
  that $f g \in \Sos_{X,2j+2k}$.
\end{corollary}

\begin{proof}
  When $j + k \geq \ee_1(X) + 1$, we have
  $H^1\bigl( \PP^, \mathcal{I}_X(j+k) \bigr) = (M_X)_{j+k} = 0$; compare with
  Remark~\ref{rem:ACM}.  The Hilbert polynomial of $X$ equals
  $\HP_X(i) = di + (1-\genus)$ and satisfies
  \[
  \HF_X(i) - \HP_X(i) = \dim H^2\bigl(\PP^2, \mathcal{I}_X(i) \bigr) - \dim
  H^1\bigl(\PP^2, \mathcal{I}_X(i) \bigr) \, ,
  \]
  so we see that $\HF_X(i) = \HP_X(i)$ for all
  $i > \max\{ \ee_1(X), \ee_2(X) \}$ and $\HF_X(i) \geq \HP_X(i)$ for all
  $i > \ee_1(X)$.  Hence, if $k \geq \ee_1(X) + 1$ and
  $2j +2k \geq \ee_2(X) + 1$, then the inequality $k \geq \frac{2 \genus}{d}$
  or $k > \frac{2 \genus -1}{d}$ is equivalently to
  $\HP_X(2j+2k) < 2 \HP_X(j+k) + \HP_X(k) - 1$ and
  $\HF_X(2j+2k) < 2 \HF_X(j+k) + \HF_X(k) - 1$.  Therefore,
  Theorem~\ref{thm:upperBound} establishes the corollary.
\end{proof}

For a second version, set $\rr(X) \coloneq \min\{ j \in \ZZ :
\text{$\HF_X(i) = \HP_X(i)$ for all $i \geq j$} \}$ where $\HP_X(i)$ denotes
the Hilbert polynomial of $X$.  This numerical invariant is sometimes called
the Hilbert regularity of $X$ or the index of regularity for $X$.  A curve
$X \subset \PP^n$ is nondegenerate if it is not contained in a hyperplane.

\begin{corollary}
  \label{cor:hilbertRegularity}
  Fix a positive integer $j$ and a nonnegative integer $k$.  Let
  $X \subset \PP^n$ be a nondegenerate totally-real curve of degree $d$ and
  arithmetic genus $\genus$, and assume that
  $k \geq \max\bigl\{ \rr(X), \frac{2 \genus}{d} \bigr\}$.  For all
  $f \in \Pos_{X,2j}$, there is a nonzero $g \in \Sos_{X,2k}$ such that
  $fg \in \Sos_{X,2j+2k}$.
\end{corollary}

\begin{proof}
  The inequalities $k \geq \rr(X)$ and $k \geq \frac{2 \genus}{d}$ yield
  $\HF_X(j+k) = \HP_X(j+k) = d(j+k) + (1-\genus)$ and
  $d(j+k) - \genus = dj + dk - \genus \geq dj + \genus \geq 2$ respectively.
  By hypothesis, the line bundle $\mathcal{O}_X(1)$ is very ample and
  $j+k \geq 1$, so the complete linear series $|\mathcal{O}_{X}(j+k)|$ defines
  a closed immersion $\varphi \colon X \to \PP^{d(j+k) - \genus}$.  If
  $Y \coloneq \varphi(X)$, then a generic hyperplane section $Y'$ of the curve $Y$
  consists of $d(j+k)$ points, any $d(j+k) - \genus$ of which are linearly
  independent; see the General Position Theorem on page~109 in \cite{ACGH}.
  Employing the inequality $k \geq \frac{2 \genus}{d}$ as second time, we
  observe that
  $d(j+k) < d(j+k) + (dk - 2 \genus) + dj - 1 = 2 \bigl( d(j+k) - \genus - 1
  \bigr) + 1$.
  Hence, the Lemma on page~115 in \cite{ACGH} establishes that the points in
  $Y'$ impose independent conditions on homogeneous polynomials of degree $2$,
  and Corollary~4.7 in \cite{EisSyz} shows that $\mathcal{I}_{Y'}$ is
  $2$-regular.  In particular, we obtain
  $H^1\bigl( \PP^{d(j+k) - \genus}, \mathcal{I}_{Y'}(1) \bigr) = 0$.  Since
  $X$ is nondegenerate, the curve $Y$ is also nondegenerate and we also have
  $H^1( \PP^{d(j+k) - \genus}, \mathcal{I}_{Y}) = 0$.  The long exact sequence
  in cohomology arising from the short exact sequence
  $0 \longrightarrow \mathcal{I}_{Y}(-1) \longrightarrow \mathcal{I}_Y
  \longrightarrow \mathcal{I}_{Y'} \longrightarrow 0$
  implies that
  $0 = H^1\bigl( \PP^{d(j+k) - \genus}, \mathcal{I}_{Y}(1) \bigr) = H^1\bigl(
  \PP^{n}, \mathcal{I}_{X}(j+k) \bigr)$.
  Since $k \geq \rr(X)$, we have $\HF_X(i) = \HP_X(i)$ for all $i \geq k$ and
  the inequality $k \geq \frac{2 \genus}{d}$ is equivalent to
  $\HF_X(2j+2k) < 2 \HF_X(j+k) + \HF_X(k) - 1$, as in the proof of
  Corollary~\ref{cor:tripleMax}.  Therefore, Theorem~\ref{thm:upperBound}
  establishes the corollary.
\end{proof}

We illustrate these corollaries for two classic families of curves.

\begin{example}[Complete intersection curves]
  \label{exa:CI}
  Consider a totally-real complete intersection curve $X \subseteq \PP^n$ cut
  out by forms of degree $d_1, d_2, \dotsc, d_{n-1}$ for $1 \leq i \leq n-1$
  where at least one $d_i$ is greater than $1$.  This curve is arithmetically
  Cohen--Macaulay, so $\ee_1(X) = - \infty$; compare with
  Remark~\ref{rem:ACM}.  By breaking the minimal free resolution of
  $\mathcal{I}_X$ (which is a Koszul complex) into short exact sequences and
  knowing the cohomology of line bundles on projective space, we deduce that
  $\rr(X) = \ee_2(X) = d_1 + d_2 + \dotsb + d_{n-1}- n - 1$.  As in
  Example~1.5.1 in \cite{Migliore}, the degree of $X$ is
  $d_1 d_2 \dotsb d_{n-1}$ and the arithmetic genus is
  $\tfrac{1}{2} ( d_1 d_2 \dotsb d_{n-1})( d_1 + d_2 + \dotsb + d_{n-1} -n-1 )
  +1$.
  Assuming that $k \geq d_1 + d_2 + \dotsb + d_{n-1} -n$,
  Corollary~\ref{cor:tripleMax} or Corollary~\ref{cor:hilbertRegularity}
  establish that, for all $f \in \Pos_{X,2j}$, there exists a nonzero
  $g \in \Sos_{X,2k}$ such that $fg \in \Sos_{X,2j+2k}$.  \hfill $\diamond$
\end{example}

\begin{example}[Planar curves]
  \label{exa:planar}
  If $X \subset \PP^2$ is a planar curve of degree $d$ at least $2$ and
  $k \geq d-2$, then Example~\ref{exa:CI} implies that, for all
  $f \in \Pos_{X,2j}$, there is a nonzero $g \in \Sos_{X,2k}$ such that
  $fg \in \Sos_{X,2j+2k}$.  \hfill $\diamond$
\end{example}

Although Example~\ref{exa:deltoid} shows that this degree bound from
Corollary~\ref{cor:hilbertRegularity} is sharp on some planar curves, the next
example demonstrates that this is not always the case.  Moreover, it
illustrates how our techniques yield sharper bounds when the convex algebraic
geometry of the underlying variety is well understood.

\begin{example}[Non-optimality for planar curves]
  \label{exa:nonoptimal}
  Let $X \subset \PP^2$ be a rational quartic curve with a real
  parametrization and a real triple point.  For instance, the curve $X$ could
  be the image of the map
  $[x_0 : x_1] \mapsto [x_0^2 x_1^{} (x_0^{} - x_1^{}) : x_0^{} x_1^2 (x_0^{}
  - x_1^{}) : x_0^4 + x_1^4]$ where $[1:0], [0:1], [1:1] \in \PP^1$ are all
  sent to $[0:0:1] \in \PP^2$; this curve has degree $4$, arithmetic genus
  $3$, and $\rr(X) = 2$.  We claim that, for all $f \in \Pos_{X,2}$, there
  exists a nonzero $g \in \Sos_{X,2}$ such that $fg \in \Sos_{X,4}$.

  We first reduce the claim to showing that a generic linear functional
  $\ell \in R_4^*$ can be written as conjugate-invariant linear combination of
  at most $8$ point evaluations on $X(\CC)$.  If the claim is false, then
  there exists a linear functional $\ell \in R_4^*$ that strictly separates
  $\Sos_{X,4}$ and $f \cdot \Sos_{X,2}$.  We may assume that $\ell \in R_4^*$
  is a generic linear functional because $\Sos_{X,4}$ and $f \cdot \Sos_{X,2}$
  are pointed cones. Since $\HF_X(2) = 6$ and $\HF_X(1) = 3$, the affine hulls
  of $\Sos_{X,4}$ and $f \cdot \Sos_{X,2}$ have dimension $6$ and $3$
  respectively.  As analysis of symmetric forms arising from point evaluations
  appearing in the proof of Proposition~\ref{pro:zero} indicates, the number
  of real point evaluations with positive coefficients plus the number of
  pairs of complex point evaluations is at least $6$ and the number of real
  point evaluations with negative coefficients plus the number of pairs of
  complex point evaluations is at least $3$.  However, if $\ell \in R_4^*$ is
  a conjugate-invariant linear combination of at most $8$ point evaluations,
  then we obtain a contradiction.

  It remains to show that a generic linear functional $\ell \in R_4^*$ is a
  conjugate-invariant linear combination of at most $8$ point evaluations on
  $X(\CC)$.  The curve $X$ is a projection of the rational normal quartic
  curve $\breve{X} \subset \PP^4$.  It follows that there is a linear
  surjection $\rho \colon \breve{R}_4^* \to R_4^*$ sending point evaluations
  on $\breve{X}$ to point evaluations on $X$.  Hence, it suffices to prove
  that, for a generic $\ell \in R_4^*$, there exists a linear functional
  $\breve{\ell} \in \rho^{-1}(\ell)$ that is a conjugate-invariant linear
  combination of at most $8$ points evaluations on $\breve{X}(\CC)$.

  By construction, the $\RR$-vector space $\breve{R}_4$ is isomorphic to
  $\RR[x_0,x_1]_{16}^{}$.  Thus, a generic linear functional
  $\breve{\ell} \in \breve{R}_4$ can be written as a conjugate-invariant
  linear combination of at most $9$ point evaluations; see Lemma~1.33 in
  \cite{IK}.  Moreover, $\breve{\ell} \in \breve{R}_4$ can be written as a
  conjugate-invariant linear combination of $9$ point evaluations if and only
  if the corresponding $(8 \times 8)$-catelecticant matrix is invertible; see
  Theorem~1.44 or the second paragraph on page~28 in \cite{IK}. A general
  element in $\breve{R}_4$, for which the corresponding
  $(8 \times 8)$-catelecticant matrix is not invertible, is a
  conjugate-invariant linear combination of $8$ point evaluations.  Hence, it
  is enough to show that there exists $\breve{\ell} \in \rho^{-1}(\ell)$ for
  which the corresponding catelecticant is not invertible.  Since three points
  of $\breve{X}$ are mapped to the same point in $X$, there exists a linear
  functional $\breve{\ell}' \in \rho^{-1}(0)$ such that the corresponding
  $(8 \times 8)$-catelecticant matrix has rank $3$; compare with Theorem~1.43
  in \cite{IK}.  Choose an arbitrary linear functional
  $\breve{\ell}'' \in \rho^{-1}(\ell)$ and consider the pencil
  $\breve{\ell}'' + \lambda \breve{\ell}'$ where $\lambda \in \RR$.  The
  determinant of the $(8 \times 8)$-catelecticant matrix corresponding to
  $\breve{\ell}'' + \lambda \breve{\ell}'$ is a polynomial of degree $3$ in
  $\lambda$.  Since every real polynomial of degree $3$ has at least one real
  root, we conclude that there is a value for $\lambda \in \RR$ such that the
  linear functional
  $\breve{\ell} \coloneq \breve{\ell}'' + \lambda \breve{\ell}'$ is a
  conjugate-invariant linear combination of at most $8$ points evaluations on
  $\breve{X}(\CC)$.  \hfill $\diamond$
\end{example}

For a nondegenerate curve, we also give a uniform bound depending only on the
degree.

\begin{corollary}
  \label{cor:degreeOnly}
  Fix a positive integer $j$ and a nonnegative integer $k$.  Let
  $X \subset \PP^n$ be a nondegenerate totally-real curve of degree $d$, and
  assume that $k \geq d-n+1$.  For all $f \in \Pos_{X,2j}$, there exists a
  nonzero $g \in \Sos_{X,2k}$ such that $f g \in \Sos_{X,2j+2k}$.
\end{corollary}

\begin{proof}
  Theorem~1.1 in \cite{GLP} proves that the Castelnuovo--Mumford regularity of
  $\mathcal{I}_X$ is at most $d-n+2$, so we have $\ee_1(X) \leq d-n$ and
  $\ee_2(X) \leq d-n-1$.  Theorem~3.2 in \cite{Nagel} establishes that
  \[
  \genus \leq \begin{cases}
    \binom{d-2}{2} -(n-3) & \text{if $d \geq 3$} \\
    2-n & \text{if $d = 2$,}
  \end{cases}
  \]
  from which we conclude that $\tfrac{2 \genus}{d} \leq d-n+1$.  Thus, the
  claim follows from Corollary~\ref{cor:tripleMax}.
\end{proof}

When the Hilbert functions of iterated hypersurface sections can be
controlled, the techniques used to prove Theorem~\ref{thm:upperBound} also
apply to higher-dimensional varieties. If a homogeneous polynomial is strictly
positive on a totally-real variety, then the associated hypersurface section
has no real points.  Focusing on non-totally-real projective varieties is
unexpectedly the key insight needed to establish our higher-dimensional
results.

\begin{lemma}
  \label{lem:nontotallyreal}
  Fix a positive integer $j$, let $X \subseteq \PP^n$ be an $m$-dimensional
  variety that is not totally real, and assume that $X$ is arithmetically
  Cohen--Macaulay.  If $\HF_X(2j) < \binom{m+2}{1} \HF_X(j) - \binom{m+2}{2}$,
  then the cone $\Sos_{X,2j}$ contains a line.
\end{lemma}

% Reusing the method from the proof of Theorem~\ref{thm:upperBound}, we prove
% this lemma by reducing to points.

\begin{proof}
  To obtain a contradiction, suppose that the cone $\Sos_{X,2j}$ contains no
  lines.  Hence, Remark~\ref{rem:nonclosed} shows that $\Sos_{X,2j}$ is
  pointed.  We begin by proving that there exist
  $h_1, h_2, \dotsc, h_m \in S_j$ such that
  $Z \coloneq X \cap \variety(h_1,h_2,\dotsc,h_m)$ is a reduced set of non-real
  points with the uniform position property.  To achieve this, observe that
  Theorem~\ref{thm:Bertini} implies that the set of homogeneous polynomials
  $h \in S_{j}$, such that $h \not\in I_X$,
  $X' \coloneq X \cap \variety(h) \subset \PP^n$, and the cone $\Sos_{X',2j}$ is
  pointed, contains a nonempty Euclidean open subset
  $U_1 \subseteq \PP(R_{j})$. Next, Bertini's Theorem (see Th\'eor\`eme~6.3 in
  \cite{J}) establishes that a general hypersurface section of a geometrically
  integral variety of dimension at least $2$ is geometrically integral and
  that a general hypersurface section of a geometrically reduced variety is
  geometrically reduced.  Thirdly, the hypothesis that $X$ is arithmetically
  Cohen--Macaulay implies that $X'$ is also arithmetically Cohen--Macaulay and
  $\HF_{X'}(k) = \HF_X(k) - \HF_X(k-j)$ for all $k \in \ZZ$.  Finally, a
  general hypersurface section of non-totally-real variety is also not
  totally real, and a general hypersurface section of a non-totally-real curve
  consists of non-real points.  Combining these four observations, we deduce
  that there exist homogeneous polynomials $h_1, h_2, \dotsc, h_m \in S_j$
  such that the intersection $Z \coloneq X \cap \variety(h_1,h_2, \dotsc, h_{m-1})$
  has the desired properties.  As the cone $\Sos_{Z,2j}$ is pointed,
  Proposition~\ref{pro:zero}~(i) now shows that
  $\HF_{Z}(j) \leq \lceil \tfrac{1}{2} \HF_{Z}(2j) \rceil$ which yields
  $2 \HF_{Z}(j) \leq \HF_{Z}(2j) + 1$.  Since we have both
  $\HF_Z(j) = \HF_X(j) - m$ and
  $\HF_Z(2j) = \HF_X(2j) - m \HF_X(j) + \binom{m}{2}$, it follows that
  $ \binom{m+2}{1} \HF_X(j) - \binom{m+2}{2} \leq \HF_X(2j)$ which gives the
  required contradiction.
\end{proof}

The inequality in Lemma~\ref{lem:nontotallyreal} has an elegant restatement
in terms of the Artinian reduction of $X$.

\begin{remark}
  If $\HF_{Z'} \colon \ZZ \to \ZZ$ is the Hilbert function of the Artinian
  quotient of $R$ by a maximal regular sequence of degree $j$, then we have
  $\HF_{Z'}(k) = \HF_Z(k) - \HF_{Z}(k-j)$ where $Z$ is the arithmetically
  Cohen--Macaulay variety defined in the antepenultimate sentence of the proof
  of Lemma~\ref{lem:nontotallyreal}.  Hence, the inequality
  $\HF_X(2j) < \binom{m+2}{1} \HF_X(j) - \binom{m+2}{2}$ is equivalent to the
  inequality $\HF_{Z'}(2j) < \HF_{Z'}(j)$.
\end{remark}

In a special case, the inequality in Lemma~\ref{lem:nontotallyreal} may also
be expressed in terms other of invariants.

\begin{remark}
  If $X \subseteq \PP^n$ is nondegenerate, then we have $\HF_X(1) = n+1$.
  Lemma~3.1 in \cite{BSV} establishes that the quadratic deficiency
  $\varepsilon(X)$ equals $\HF_X(2) - (m+1)(n+1) + \binom{m+1}{2}$.  Hence,
  the addition formula for binomial coefficients gives
  \begin{align*}
    \HF_X(2) - \tbinom{m+2}{1} \HF_X(1) + \tbinom{m+2}{2} 
    &= \bigl[ \HF_X(2) - \tbinom{m+1}{1} \HF_X(1) + \tbinom{m+1}{2} \bigr] -
      \bigl[ \tbinom{m+1}{0} \HF_X(1) - \tbinom{m+1}{1} \bigr] \\
    &= \varepsilon(X) - \codim(X) \, ,
  \end{align*}
  so the inequality in Lemma~\ref{lem:nontotallyreal} becomes
  $\varepsilon(X) < \codim(X)$ when $X$ is nondegenerate and $j = 1$.
\end{remark}

Lemma~\ref{lem:nontotallyreal} shows that there exists a nontrivial sum of
squares equal to zero.  Exploiting this observation, we can prove a
higher-dimensional analogue of Theorem~\ref{thm:upperBound}.

\begin{theorem}
  \label{thm:totallyreal}
  Fix a positive integer $j$ and a nonnegative integer $k$.  Let
  $X \subseteq \PP^n$ be a totally-real variety with dimension $m$. Assume
  that $X$ is arithmetically Cohen--Macaulay and that
  \[
  \HF_X(2j+2k) < \tbinom{m+1}{1} \bigl( \HF_X(j+k) - \HF_X(k-j) \bigr) +
  \HF_X(2k) - \tbinom{m+1}{2} \, .
  \] 
  For all $f \in \Pos_{X,2j}$, there exists a nonzero $g \in R_{2k}$ such that
  $fg \in \Sos_{X,2j+2k}$.
\end{theorem}

\begin{proof}
  With the aim of producing a contradiction, suppose that, for all nonzero
  $g \in R_{2k}$, we have $fg \not\in \Sos_{X,2j+2k}$.  This means that the
  linear subspace
  $f \cdot R_{2k} \coloneq \{ fg \in R_{2j+2k} : g \in R_{2k} \} \subset
  R_{2j+2k}$ intersects the cone $\Sos_{X,2j+2k}$ only at the origin.  As $X$
  is totally real, Corollary~\ref{cor:variety} establishes that the cone
  $\Sos_{X,2j+2k}$ is pointed and, in particular, closed.  Hence, there exists
  a Euclidean open neighbourhood $U$ of $f \in R_{2j}$ such that, for all
  $h \in U$ and all nonzero $g \in R_{2k}$, we have
  $hg \not\in \Sos_{X,2j+2k}$.  Bertini's Theorem (see Th\'eor\`eme~6.3 in
  \cite{J}) establishes that a general hypersurface section of a geometrically
  reduced variety is geometrically reduced.  The cone $\Pos_{X,2j}$ is
  full-dimensional, so there is a general $h \in U \cap \Pos_{X,2j}$ such that
  $X' \coloneq X \cap \variety(h)$ and
  $R' = R/\langle h \rangle = S/(I_X + \langle h \rangle)$.  Every real zero
  of $h$ must be contained in the singular locus of $X'$ because
  $h \in \Pos_{X,2j}$.  As $X'$ is reduced, its singular locus is a proper
  Zariski closed subset, which implies that $X'$ is not totally real.  Since
  $X$ is arithmetically Cohen--Macaulay, the variety $X'$ is also
  arithmetically Cohen--Macaulay and $\HF_{X'}(i) = \HF_X(i) - \HF_X(i-2j)$.
  From
  $\HF_X(2j+2k) < \tbinom{m+1}{1} \bigl( \HF_X(j+k) - \HF_X(k-j) \bigr) +
  \HF_X(2k) - \tbinom{m+1}{2}$, we obtain
  $\HF_{X'}(2j+2k) < \binom{(m-1)+2}{1} \HF_{X'}(j+k) - \binom{(m-1)+2}{2}$.
  Hence, Lemma~\ref{lem:nontotallyreal} shows that the cone $\Sos_{X',2j+2k}$
  contains a line.  Applying Proposition~\ref{pro:pointed}, there exist
  nonzero $g_1', g_2', \dotsc, g_s' \in R_{j+k}'$ such that
  $(g_1')^2 + (g_2')^2 + \dotsb + (g_s')^2 = 0$.  Lifting this equation to the
  ring $R$, we see that there are $g_1,g_2,\dotsc,g_s \in R_{j+k}$ such that
  $g_1^2 + g_2^2 + \dotsb + g_s^2 \in \langle h \rangle$.  However, this
  contradicts the fact that $h \in U$.  Therefore, we conclude that there
  exists a nonzero $g \in R_{2k}$ such that $fg \in \Sos_{X,2j+2k}$.
\end{proof}

\begin{remark}
  \label{rem:nonnegative}
  Suppose that $f \in \Pos_{X,2j}$ is strictly positive on $X(\RR)$ or, more
  generally, that the subset $X(\RR) \setminus \variety(f)$ is dense in the
  Euclidean topology.  For instance, the second condition automatically holds
  when $f$ is nonzero and $X(\RR)$ a cone over a manifold in which all of the
  connected components have the same dimension.  With this extra hypothesis,
  the nonzero multiplier $g \in R_{2k}$ described in
  Theorem~\ref{thm:totallyreal} must be nonnegative.
\end{remark}

\begin{remark}
  \label{rem:nontrivial}
  Suppose that $f \in \Pos_{X,2j}$ is strictly positive on $X(\RR)$.  If the
  degree of the nonzero multiplier $g \in R_{2k}$ to be greater than or equal
  to the degree of $f \in \Pos_{X,2j}$, then one obtains a frivolous
  sum-of-squares representation $fg = f^2h \in \Sos_{X,2j+2k}$ by choosing
  $g \coloneq fh$ where $h \in \Sos_{X,2k-2j}$.  However, the products
  $fg \in \Sos_{X,2j+2k}$ arising from Theorem~\ref{thm:totallyreal} never
  have this frivolous form because Lemma~\ref{lem:nontotallyreal} shows that
  they are lifted from a nontrivial sum-of-squares modulo $f$.
\end{remark}

The next four examples showcase the most interesting applications of
Theorem~\ref{thm:totallyreal}.  In these examples, we also obtain simple
explicit degree bounds on the sum-of-squares multipliers.

\begin{example}[Nonnegative quadratic forms on varieties of minimal degree]
  \label{exa:quadraticForms}
  Fix $j = 1$ and $k = 0$.  Let $X \subseteq \PP^n$ be a totally-real variety
  of minimal degree.  In other words, the variety $X$ is nondegenerate and
  $\deg(X) = 1 + \codim(X) = 1+n-m$ where $m \coloneq \dim(X)$.  The classification
  of varieties of minimal degree (see Theorem~1 in \cite{EH}) implies that $X$
  is arithmetically Cohen--Macaulay and
  $\sum_{i \in \ZZ} \HF_X(i) t^i = \bigl(1+ (n-m) t \bigr)(1-t)^{-(m+1)}$.
  Hence, the Generalized Binomial Theorem establishes that
  $\HF_X(i) = \binom{m+i}{m} + (n-m) \binom{m+i-1}{m}$ for $i \geq 1-m$.  It
  follows that
  \begin{multline*}
    \tbinom{m+1}{1} \bigl( \HF_X(j+k) - \HF_X(k-j) \bigr) + \HF_X(2k) -
    \tbinom{m+1}{2} - \HF_X(2j+2k) \\
    = \tbinom{m+1}{1} \Bigl( \tbinom{m+1}{m} + (n-m) \tbinom{m}{m} \Bigr) + 1
    - \tbinom{m+1}{2} - \Bigl( \tbinom{m+2}{m} + (n-m) \tbinom{m+1}{m} \Bigr)
    = 1 > 0 \, ,
  \end{multline*}
  so Theorem~\ref{thm:totallyreal} shows that $\Pos_{X,2} = \Sos_{X,2}$.  This
  gives another proof of Proposition~4.1 in \cite{BSV}.  \hfill $\diamond$
\end{example}

\begin{example}[Nonnegative forms on surfaces of minimal degree]
  \label{exa:minimalSurface}
  Fix $j \geq 1$ and $k = j-1$.  Let $X \subseteq \PP^n$ be a totally-real
  surface of minimal degree.  As in Example~\ref{exa:quadraticForms}, the
  variety $X$ is arithmetically Cohen--Macaulay, and we have
  $\HF_X(i) = \binom{i+2}{2} +(n-2) \binom{i+1}{2}$ for $i \geq -1$.  Since
  \begin{multline*}
    \tbinom{2+1}{1} \bigl( \HF_X(j+k) - \HF_X(k-j) \bigr) + \HF_X(2k) -
    \tbinom{2+1}{2} - \HF_X(2j+2k) \\
    = 3 \bigl( \HF_X(2j-1) - \HF_X(-1) \bigr) + \HF_X(2j-2) - 3 - \HF_X(4j-2)
    = 4j - 3 > 0 \, ,
  \end{multline*}
  Theorem~\ref{thm:totallyreal} shows that, for all $f \in \Pos_{X,2j}$, there
  exists a nonzero $g \in R_{2j-2}$ such that $fg \in \Sos_{X,4j-2}$.
  Remark~\ref{rem:nonnegative} also implies that $g \in \Pos_{X,2j-2}$.
  Because Example~\ref{exa:quadraticForms} proves that we have
  $g \in \Sos_{X,2j-2}$ when $j = 1$, an induction on $j$ establishes that,
  for all $f \in \Pos_{X,2j}$, there exists a nonzero $h \in \Sos_{X,j^2-j}$
  such that $fh \in \Sos_{X,j^2+j}$.  \hfill $\diamond$
\end{example}

\begin{example}[Nonnegative forms on the projective plane]
  \label{exa:formsOnPP2}
  Fix $j \geq 2$ and $k = j-2$.  The variety $\PP^2$ is arithmetically
  Cohen--Macaulay and $\HF_{\PP^2}(i) = \binom{i+2}{2}$ for $i \geq -2$.  It
  follows that
  \begin{multline*}
    \tbinom{2+1}{1} \bigl( \HF_{\PP^2}(j+k) - \HF_{\PP^2}(k-j)
    \bigr) + \HF_{\PP^2}(2k) - \tbinom{2+1}{2} - \HF_{\PP^2}(2j+2k) \\
    = 3 \bigl( \HF_{\PP^2}(2j-2) - \HF_{\PP^2}(-2) \bigr) + \HF_{\PP^2}(2j-4)
    - 3 - \HF_{\PP^2}(4j-4) = 2j-3 > 0 \, ,
  \end{multline*}
  so Theorem~\ref{thm:totallyreal} and Remark~\ref{rem:nonnegative} show that,
  for all $f \in \Pos_{\PP^2,2j}$, there exists a nonzero
  $g \in \Pos_{\PP^2,2j-4}$ such that $fg \in \Sos_{\PP^2,4j-4}$.  In
  particular, this re-establishes a result of Hilbert (see \cite{Hilbert2} or
  Theorem~2.6 in \cite{Blekherman}).  As in Example~\ref{exa:minimalSurface},
  an induction on $j$ proves that
  \begin{itemize}[$\bullet$]
  \item for all $f \in \Pos_{\PP^2,4j}$, there exists a nonzero
    $h \in \Sos_{\PP^2,2j^2 - 2j}$ such that $fh \in \Sos_{\PP^2,2j^2+2j}$, and
  \item for all $f \in \Pos_{\PP^2,4j-2}$, there exists a nonzero
    $h \in \Sos_{\PP^2,2j^2-4j+2}$ such that $fh \in \Sos_{\PP^2,2j^2}$.
  \end{itemize}
  Since $\Pos_{\PP^2,6} \neq \Sos_{\PP^2,6}$, this degree bound is sharp for
  $f \in \Pos_{\PP^2,6}$ and Example~\ref{exa:strictPP^2} shows that it is
  also sharp for $f \in \Pos_{\PP^2,8}$. \hfill $\diamond$
\end{example}

\begin{example}[Nonnegative forms on some surfaces of almost minimal degree]
  \label{exa:almost}
  Fix $j \geq 1$ and $k = j$.  Let $X \subset \PP^n$ be a totally-real surface
  that is arithmetically Cohen--Macaulay and satisfies
  $\sum_{i \in \ZZ} \HF_X(i) t^i = \bigl( 1 + c_1t + c_2t^2 + c_3t^3 + c_4t^4
  \bigr) (1-t)^{-3}$
  for some $c_1, c_2, c_3, c_4 \in \QQ$.  The Generalized Binomial Theorem
  yields
  $\HF_X(i) = \binom{i+2}{2} + c_1\binom{i+1}{2} + c_2\binom{i}{2} + c_3
  \binom{i-1}{2} + c_4 \binom{i-2}{2}$ for all $i \geq 2$, so it follows that
  \begin{multline*}
    \tbinom{2+1}{1} \bigl( \HF_X(j+k) - \HF_X(k-j) \bigr) + \HF_X(2k) -
    \tbinom{2+1}{2} - \HF_X(2j+2k) \\
    = 3 \bigl( \HF_X(2j) - \HF_X(0) \bigr) + \HF_X(2j) - 3 - \HF_X(4j) =
    2(c_1-c_2-3c_3-5c_4+3)j +3(c_3+3c_4-1) \, .
  \end{multline*}
  Thus, if $2c_1+3 > 2c_2+3c_3+c_4$, then Theorem~\ref{thm:totallyreal} shows
  that, for all $f \in \Pos_{X,2j}$, there exists a nonzero $g \in R_{2j}$
  such that $fg \in \Sos_{X,4j}$.  For instance, if $X$ is a totally-real
  surface of almost minimal degree that is arithmetically Cohen--Macaulay (in
  other words, the surface $X$ is nondegenerate, arithmetically
  Cohen-Macaulay, and $\deg(X) = 2 + \codim(X) = n$), then we have
  $c_1 = n-2$, $c_2 = 1$, $c_3 = 0$, and $c_4 = 0$, which implies that
  $2c_1+3 = 2n-1 > 2 = 2c_2+3c_3+c_4$.  By Remark~\ref{rem:nontrivial}, this
  certificate is not frivolous.  \hfill $\diamond$
\end{example}

%%%%%%%%%%%%%%%%%%%%%%%%%%%%%%%%%%%%%%%%%%%%%%%%%%%%%%%%%%%%%%%%%%%%%%%%%%%%%%
\section{Lower bounds for Sum-of-Squares Multipliers}
\label{sec:boundingSeparators}

\noindent
This final section establishes lower bounds on the minimal degree of a
sum-of-squares multiplier.  These degree bounds for the non-existence of
sum-of-squares multipliers prove the second halves of our main theorems.  For
Harnack curves on smooth toric surfaces, these degree bounds for the existence
of strict-separators are a perfect complement to our degree bounds for the
existence of sum-of-squares multipliers.

Our first lemma relates the zeros of a nonnegative element to the zeros of any
sum-of-squares multiplier.  For a closed point $p \in X$, let
$\der_p \colon R \to \textup{T}_p^*(X)$ denote the derivation that sends
$f \in R$ to the class of $f-f(p)$ in the Zariski cotangent space at $p$.

\begin{lemma}
  \label{lem:derivation}
  Fix a positive integer $j$ and a nonnegative integer $k$.  Let
  $X \subseteq \PP^n$ be a totally-real projective variety, and consider
  $f \in \Pos_{X,2j}$ and $g \in \Sos_{X,2k}$ such that
  $fg \in \Sos_{X,2j+2k}$.  If the real point $p \in X(\RR)$ satisfies
  $f(p) = 0$ and $\der_p(f) \neq 0$, then we have $g(p) = 0$ and
  $\der_p(g) = 0$.
\end{lemma}

\begin{proof}
  Suppose $fg = h$ where $h \coloneq h_0^2 + h_2^2 + \dotsb + h_s^2$ for
  $h_0, h_1, \dotsc, h_s \in R_{j+k}$.  Since $f(p) = 0$, we see that
  $h(p) = 0$ and $h_j(p) = 0$ for all $0 \leq j \leq s$.  Hence, the Leibniz
  Rule establishes that
  $0 = 2h_0(p) \der_p(h_0) + 2 h_1(p) \der_p(h_1) + \dotsb + 2 h_s(p)
  \der_p(h_s) = \der_p(h) = f(p) \der_p(g) + g(p) \der_p(f)$.
  By hypothesis, we have $f(p) = 0$ and $\der_p(f) \neq 0$, which implies that
  $g(p) = 0$.  Since $g$ is a sum of squares, we conclude that
  $\der_p(g) = 0$, as we just did for $h$.
\end{proof}

\begin{remark}
  When $f \in \Pos_{X,2j}$, the hypothesis that $\der_p(f) \neq 0$ can only be
  satisfied if $p$ is a singular point on the variety $X$.
\end{remark}

Equipped with this lemma, we show that there exists a planar curve for which
the bound on the degree of multipliers given in Example~\ref{exa:planar} is
tight.

\begin{example}[Optimality for a planar curve]
  \label{exa:deltoid}
  Let $X \subset \PP^2$ be the rational tricuspidal quartic curve defined by
  the equation
  $(x_0^2 + x_1^2)^2 + 2x_2^2(x_0^2 + x_1^2) - \frac{1}{3} x_2^4 - \frac{8}{3}
  x_2^{} (x_0^3 - 3 x_0^{} x_1^2) = 0$.  This curve is called the deltoid
  curve and is parametrized by
  $t \mapsto \bigl[ \frac{1}{3} \bigl( 2 \cos(t) + \cos(2t) \bigr) :
  \frac{1}{3} \bigl( 2 \sin(t) - \sin(2t) \bigr) : 1 \bigr]$ in the affine
  plane $x_2 = 1$.  In other words, the real points of $X$ consist of the
  hypocycloid generated by the trace of a fixed point on a circle that rolls
  inside a larger circle with one-and-a-half times its radius. The three cusps
  occur at the points $\bigl[ 1 : 0 : 1 \bigr]$,
  $\bigl[- \tfrac{1}{2} : \tfrac{\sqrt{3}}{2} : 1 \bigr]$,
  $\bigl[- \tfrac{1}{2} : -\tfrac{\sqrt{3}}{2} : 1 \bigr]$, corresponding to
  $t = 0, \tfrac{2 \pi}{3}, \tfrac{4 \pi}{3}$ respectively, and lie on the
  conic $x_2^2 - x_1^2 - x_0^2$.

  Consider $f \in \Pos_{X,2j}$ such that $\der_p(f) \neq 0$ at each cusp $p$
 in $X$.  For instance, the polynomial
  $(x_2^2 - x_1^2 - x_0^2)(x_0^2 + x_1^2 + x_2^2)^{j-1}$ is nonnegative on $X$
  and has nonzero derivations at each cusp $p$ on $X$.  Suppose that there
  exists a nonzero $g \in \Sos_{X,2}$ such that $fg \in \Sos_{X,2j+2}$.
  Lemma~\ref{lem:derivation} implies that $g(p) = 0$ and $\der_p(g) = 0$ at
  each cusp $p$ of $X$.  Expressing $g$ as a sum of linear forms, it follows
  that each of these linear forms vanishes at all three cusps.  Since the
  three cusps are not collinear, this is impossible.  Therefore, for all
  nonzero $g \in \Sos_{X,2}$, we conclude that $fg \not\in \Sos_{X,2j+2}$.
  \hfill $\diamond$
\end{example}

We next examine rational curves on a projective surface.  A \define{surface}
$Y \subseteq \PP^n$ is a two-dimensional projective variety;  for more
information on algebraic surfaces, see \cite{Beauville}.

\begin{lemma}
  \label{lem:nonnegativeExistence}
  Let $Y \subseteq \PP^n$ be a real surface and let $X$ be a curve on $Y$.  If
  $X$ has $j$ isolated real points $p_1, p_2, \dotsc, p_j$, then there exists 
  $f \in \Pos_{X,2j}$ such that $f(p_i) = 0$ and $\der_{p_i}(f) \neq 0$ for
  all $1 \leq i \leq j$.
\end{lemma}

\begin{proof}
  Fix coordinates on $\PP^n$ such that the hyperplane $\variety(x_0)$ does not
  contain any isolated real points on $X$.  For each isolated singular point
  $p_i \in X(\RR)$ where $1 \leq i \leq j$, let
  $\tilde{p}_i \in \AA^{n+1}(\RR)$ be the affine representative in which the
  $0$-th component equals $1$.  Choose a real point $\tilde{q}_i$ in
  $\variety(x_0 -1) \subset \AA^{n+1}(\RR)$ such that the closed ball centered
  at $\tilde{q}_i$ with radius
  $\varepsilon_i \coloneq \norm{\tilde{p}_i-\tilde{q}_i}^2 > 0$ does not contain an
  affine representative $\tilde{p}$ where $p \in X(\RR)$ except for the point
  $\tilde{p}_i$ corresponding to an isolated real point.  For
  $1 \leq i \leq j$, consider
  $\tilde{h}_i \coloneq \bigl( x_1 - (\tilde{q}_i)_1^{} \, x_0 \bigr)^2 + \bigl( x_2
  - (\tilde{q}_i)_2^{} \, x_0 \bigr)^2 + \dotsb + \bigl( x_n -
  (\tilde{q}_i)_n^{} \, x_0 \bigr)^2 - \varepsilon_i x_0^2 \in S_2$.
  If $\tilde{h}_i$ maps to $h_i \in R_2$ under the canonical quotient map from
  $S$ to $R$, then we have $h_i \in \Pos_{X,2}$, $h_i(p_i) = 0$, and
  $\der_{p_i}(h_i) \neq 0$ by construction.  Hence, the product
  $f \coloneq h_1 h_2 \dotsb h_j \in R_{2j}$ satisfies the conditions in the first
  part of the lemma.
\end{proof}

To obtain the desired bounds, we make additional assumptions on the surface
and the curve.  On a curve, an ordinary double point (also known as a node or
an $A_1$-singularity) is a point where a curve intersects itself so that the
two branches of the curve have distinct tangent lines.  There are two types of
ordinary real double points: a \define{crossing} has two real branches and a
\define{solitary point} has two imaginary branches that conjugate to each
other.  Hence, an isolated ordinary real double point is a solitary point.
The following proposition is the basic source of our bounds for
strict-separators.

\begin{proposition}
  \label{pro:lower}
  Let $Y \subseteq \PP^n$ be a real smooth rational surface such that the
  anti-canonical divisor is effective, and let $H$ be a hyperplane section of
  $Y$.  For some positive integer $j$, assume that there exists a section in
  $H^0\bigl(Y, \mathcal{O}_Y(jH) \bigr)$ that defines a real rational curve
  $X \subset Y$ of degree $d$ and arithmetic genus $\genus$.  If $X$ has
  $\genus$ solitary points $p_1, p_2, \dotsc, p_{\genus}$, then there exists
  $f \in \Pos_{X,2j+2}$ such that $f(p_i) = 0$ and $\der_{p_i}(f) \neq 0$ for
  all $1 \leq i \leq \genus$.  Moreover, if the nonzero element
  $g \in \Sos_{X,2k}$ satisfies $fg \in \Sos_{X,2j+2k+2}$, then we have
  $k \geq \frac{2\genus}{d}$.
\end{proposition}

\begin{proof}
  Let $K$ be the canonical divisor on $Y$.  Since $X$ is projective and the
  divisor $jH$ is effective, Serre Duality (see Theorem~I.11 in
  \cite{Beauville}) shows that
  $H^2\bigl(Y, \mathcal{O}_Y(jH+K) \bigr) = H^0(Y, \mathcal{O}_Y(-jH) \bigr) =
  0$.
  As $Y$ is rational and the irregularity and geometric genus of a surface are
  birational invariants (see Proposition~III.20 in \cite{Beauville}), we have
  $H^1(Y, \mathcal{O}_Y) = 0$ and
  $H^2(Y, \mathcal{O}_Y) = H^0 \bigr(Y, \mathcal{O}_Y(K) \bigr) = 0$, so the
  Euler--Poincar\'e characteristic $\chi( \mathcal{O}_Y)$ equals $1$.
  Applying the Riemann--Roch Theorem (see Theorems I.12 and I.15 in
  \cite{Beauville}), it follows that
  \[
  \chi\bigl( \mathcal{O}_Y(jH+K) \bigr) = \chi(\mathcal{O}_Y) + \tfrac{1}{2}
  \bigl( (jH+K)^2 - (jH+K).K \bigr) = 1 + \tfrac{1}{2}\bigl( (jH)^2 +(jH).K) =
  \genus \, ,
  \]
  and we deduce that
  $\dim H^0\bigl(Y, \mathcal{O}_Y(jH+K) \bigr) \geq \genus$. 

  We first prove that the solitary points impose independent conditions by
  verifying that there is no nonzero section of $\mathcal{O}_Y(jH+K)$ which
  vanishes at any $\genus -1$ solitary points of $X$ and at any additional
  point $q \in X$. Suppose there exists a nonzero section of
  $\mathcal{O}_Y(jH+K)$ which vanishes at $\genus -1$ solitary points of $X$
  and an additional point $q \in X$.  Let $\widetilde{Y}$ be the blowing up of
  the surface $Y$ at $\genus - 1$ solitary points and the point $q$; the
  corresponding exceptional divisors are $E_1, E_2, \dotsc, E_{\genus -1}, F$.
  If this hypothetical section vanishes at the chosen $\genus - 1$ nodes of
  $X$ and the point $q \in X$ with multiplicities $m_i$ and $r$ respectively,
  then the line bundle
  \[
  \mathcal{O}_{\widetilde{Y}}(jH + K - m_1E_1 - m_2E_2 - \dotsb -
  m_{\genus-1} E_{\genus -1} - rF)
  \]
  restricted to the proper transform of $X$ in $\widetilde{Y}$ would also have
  a section.  However, the degree of the restriction (see Lemma~I.6 in
  \cite{Beauville}) equals
  \begin{multline*}
    (jH + K - m_1E_1 - m_2E_2 - \dotsc - m_{\genus-1} E_{\genus -1} - rF).(jH
    - 2E_1 - 2 E_2 - \dotsb - 2E_{\genus -1} - F) \\
    = 2(\genus -1) - 2(m_1+m_2 + \dotsb + m_{\genus -1}) -r < 0 \, ,
  \end{multline*}
  which yields the required contradiction.

  To prove the first part, choose a nonzero section
  $f_1 \in H^0\bigl(Y, \mathcal{O}_Y(jH+K) \bigr)$ that vanishes at the
  solitary points $p_2, p_3, \dotsc, p_{\genus}$.  The previous paragraph
  ensures that $f_1(p_1) \neq 0$.  Because the solitary points
  $p_1, p_2, \dotsc, p_{\genus}$ are isolated and imposed independent
  conditions, there exists a nearby section
  $f_2 \in H^0\bigl(Y, \mathcal{O}_Y(jH+K) \bigr)$, a small perturbation of
  $f_1$, that does not vanish at any point in $X(\RR)$.  Since the
  anti-canonical divisor $-K$ is effective, we may also choose a nonzero
  section $f_3 \in H^0\bigl( Y, \mathcal{O}_Y(-K) \bigr)$.  By construction,
  the section $f_1^{}f_2^{}f_3^2 \in H^0 \bigl( Y, \mathcal{O}_Y(2jH) \bigr)$
  is greater than or equal to zero at all points in
  $X(\RR) \setminus \{ p_1 \}$; see Section~5 in \cite{BSV} for more on the
  sign of a section.  Applying Lemma~\ref{lem:nonnegativeExistence}, there
  exists $f_4 \in H^0 \bigl( Y, \mathcal{O}_Y(2H) \bigr)$ such that
  $f_4(p_1) = 0$ and $\der_{p_1}(f_4) \neq 0$.  Hence, the section
  $f \coloneq f_1^{} f_2^{} f_3^2 f_4^{} \in H^0 \bigl( Y,
  \mathcal{O}_Y(2(j+1)H) \bigr)$, which is the restriction to $Y$ of a
  hypersurface of degree $2j+2$ in $\PP^n$, is nonnegative on $X$ and
  satisfies $f(p_i) = 0$ and $\der_{p_i}(f) \neq 0$ for all
  $1 \leq i \leq \genus$.

  For the second part, consider a nonzero multiplier $g \in \Sos_{X,2k}$ such
  that $fg \in \Sos_{X,2j+2k}$.  Lemma~\ref{lem:derivation} establishes that
  $g(p_i) = 0$ and $\der_{p_i}(g) = 0$ for $1 \leq i \leq j$.  Fix an element
  $\tilde{g}$ of degree $2k$ in the $\ZZ$-graded coordinate ring of $Y$ that
  maps to $g \in R_{2k}$ under the canonical quotient homomorphism and
  consider the curve $Z \subset Y$ defined by $\tilde{g}$.  Since the element
  $g$ is nonzero in $R_{2k}$, the curve $Z$ does not contain the curve $X$.
  Let $\widehat{Y}$ be the blowing up of the surface $Y$ at the $\genus$
  solitary points $p_1, p_2, \dotsc, p_{\genus}$ and let
  $E_1, E_2, \dotsc, E_{\genus}$ be the corresponding exceptional divisors in
  $\widehat{Y}$.  The proper transforms $\widehat{X} \subset \widehat{Y}$ and
  $\widehat{Z} \subset \widehat{Y}$ of the curves $X \subset Y$ and
  $Z \subset Y$ are linearly equivalent to the divisor classes
  $D_{\widehat{X}} \coloneq eH - 2E_1 -2 E_2 - \dotsb -2 E_{\genus}$ and
  $D_{\widehat{Y}} \coloneq 2kH - m_1 E_1 - m_2 E_2 - \dotsb - m_{\genus}
  E_{\genus}$ for some $m_i \geq 2$.  Since $\widehat{X}$ is irreducible, the
  degree of the line bundle $\mathcal{O}_{\widehat{X}}(D_Y)$ is nonnegative.
  Hence, we obtain
  $0 \leq D_{\widehat{Y}} . D_{\widehat{X}} = 2ekH^2 - 2(m_1 + m_2 + \dotsb +
  m_{\genus}) \leq 2 k d - 4\genus$, which yields $k \geq \frac{2 \genus}{d}$.
\end{proof}

\begin{remark}
  By modifying the third paragraph in the proof of
  Proposition~\ref{pro:lower}, one can obtain slightly better bounds when the
  canonical divisor $K$ is a multiple of the hyperplane section $H$.  In
  particular, this applies for $Y = \PP^2$.
\end{remark}

Although Proposition~\ref{pro:lower} is the latent source for our sharpness
results, it is technically difficult to apply because of its hypotheses.  To
address this challenge, we exhibit the appropriate rational curves on toric
surfaces.  To be more precise, consider a smooth convex lattice polygon
$Q \subset \RR^2$ and its associated nonsingular toric surface $Y_Q$.  Fix a
cyclic ordering for the edges of $Q$, let $u_1, u_2, \dotsc, u_m \in \ZZ^2$ be
the corresponding primitive inner normal vectors to the edges, and let
$D_1, D_2, \dotsc, D_m$ be the corresponding irreducible torus-invariant
divisors on $Y_Q$.  The anti-canonical divisor on $Y_Q$ is the effective
divisor $D_1 + D_2 + \dotsb + D_m$. From the canonical presentation for the
convex polytope $Q = \{v \in \RR^2 : \text{$\langle v,u_i \rangle \geq -a_i$
  for $1 \leq i \leq m$} \}$, we obtain the very ample divisor
$A_Q \coloneq a_1 D_1 + a_2D_2 + \dotsb + a_m D_m$ on $Y_Q$.  For more background on
toric geometry, see Section~2.3 and Section~4.2 in \cite{CLS}.

As in Subsection~2.2 in \cite{KO}, we call the real connected components of a
curve $X \subset Y_Q$ \define{ovals} and treat isolated real points as
degenerate ovals.  Following Definition~8 in \cite{Br}, a \define{Harnack
  curve} $X \subset Y_Q$ is the image of a real morphism $\xi \colon C \to
Y_Q$ satisfying three conditions:
\begin{enumerate}[\upshape (1)]
\item the smooth real curve $C$ has the maximal number of ovals (namely, one
  more than the genus of the curve $C$);
\item there is a distinguished oval in $C(\RR)$ containing disjoint arcs
  $\Gamma_1, \Gamma_2, \dotsc, \Gamma_m$ such that, for all $1 \leq i \leq m$,
  we have $\xi^{-1}(D_i) \subseteq \Gamma_j$; and
\item the cyclic orientation on the arcs induced by the distinguished oval is
  exactly $[\Gamma_1, \Gamma_2, \dotsc, \Gamma_m]$.
\end{enumerate}
These special curves are germane because Theorem~10 in \cite{Br} establishes
that all of the singularities on a Harnack curve are solitary points.  By
modifying the technique in Subsection~4.1 of \cite{KO} for $\PP^2$, we
construct rational Harnack curves on smooth projective toric surfaces.

\begin{proposition}
  \label{pro:Harnack}
  If $Q \subset \RR^2$ is a smooth two-dimensional lattice polygon, then there
  exists a rational Harnack curve on the toric variety $Y_Q$ which is linearly
  equivalent to the associated very ample divisor $A_Q$ and has arithmetic
  genus equal to the number of interior lattice points in $Q$.
\end{proposition}

\begin{proof}
  Following \cite{Cox}, a map from $\PP^1$ to the smooth toric variety
  $Y_Q$ is determined by a collection of line bundles and sections on $\PP^1$
  that satisfy certain compatibility and non-degeneracy conditions.  To
  describe the required map, fix disjoint arcs
  $\Gamma_1, \Gamma_2, \dotsc, \Gamma_m$ on the circle $\PP^1(\RR)$ such that
  the induced cyclic orientation is $[\Gamma_1, \Gamma_2, \dotsc, \Gamma_m]$.
  The intersection product $e_i \coloneq A_Q \cdot D_i$, for each $1 \leq i \leq m$,
  equals the normalized lattice distance of the corresponding edge in the
  polytope $Q$.  The Divergence Theorem
  % (applied in the plane to the constant vector field $v$) 
  shows that
  $e_1 \langle v, u_1 \rangle + e_2 \langle v, u_2 \rangle + \dotsb + e_m
  \langle v, u_m \rangle = 0$
  for all $v \in \ZZ^2$, so the line bundles
  $\mathcal{O}_{\PP^1}(e_1), \mathcal{O}_{\PP^1}(e_2), \dotsc,
  \mathcal{O}_{\PP^1}(e_m)$
  satisfy the compatibility condition in Definiton~1.1 in \cite{Cox}.  For all
  $1 \leq i \leq m$, choose distinct points
  $[c_{i,1}:1], [c_{i,2}:1], \dotsc, [c_{i,e_i}:1] \in \Gamma_i$.  Identifying
  global sections of $\mathcal{O}_{\PP^1}(e_i)$ with homogeneous polynomials in
  $\CC[x_0, x_2]_{e_i}$, we obtain the real sections
  $\prod_{j=1}^{e_i} (x_0 - c_{i,j} x_1) \in H^0 \bigl( \PP^1,
  \mathcal{O}_{\PP^1}(e_i) \bigr)$.
  Since we chose distinct points, no two sections vanish at the same point in
  $\PP^1$, so these sections satisfy the non-degeneracy condition in
  Definition~1.1 in \cite{Cox}.  Hence, Theorem~1.1 in \cite{Cox} establishes
  that these line bundles and sections determine a real morphism
  $\xi \colon \PP^1 \to Y_Q$ such that
  $\xi^{-1}(D_i) = \{ [c_{i,1}:1], [c_{i,2}:1], \dotsc, [c_{i,e_i}:1] \}$ for
  all $1 \leq i \leq m$.  In other words, the image of $\xi$ is a rational
  Harnack curve $X \subset Y_Q$.  By construction, the curve $X$ is also
  linearly equivalent to the divisor $A_Q$.  Hence, Proposition~10.5.8 in
  \cite{CLS} proves that the arithmetic genus of $X$ equals the number of
  interior lattice points in $Q$.
\end{proof}

Having assembled the necessary prerequisites, we now describe our lower bound
on the degrees of sum-of-squares multipliers on curves.

\begin{theorem}
  \label{thm:lower}
  For all $j \geq 2$, there exist smooth curves $X \subset \PP^n$ and elements
  $f \in \Pos_{X,2j}$ such that the cones $\Sos_{X,2j+2k}$ and
  $f \cdot \Sos_{X,2k}$ are well-separated for all $k < \frac{2 \genus}{d}$
  where $d$ and $\genus$ are the degree and genus of $X$ respectively.
\end{theorem}

\begin{proof}
  Fix a smooth two-dimensional lattice polytope $Q$ and let $A_Q$ be the
  associated very ample divisor on the smooth toric variety $Y_Q$.  Applying
  Proposition~\ref{pro:Harnack} to the dilated polytope $(j-1)Q$ gives a
  rational Harnack curve $X$ on $Y_Q$ of degree $d$ defined by a section in
  $H^0\bigl(Y_Q, \mathcal{O}_{Y_Q}\bigl((j-1)A_Q\bigr)\bigr)$.  The number of
  singular points on $X$ equals its arithmetic genus $\genus$ and, as
  Theorem~10 in \cite{Br} establishes, all of the singularities on $X$ are
  solitary points.  Hence, Proposition~\ref{pro:lower} shows that there exists
  an element $f \in \Pos_{X,2j}$ such that, for all $k < \frac{2 \genus}{d}$,
  the cones $\Sos_{X,2j+2k}$ and $f \cdot \Sos_{X,2k}$ are well-separated.
  Asserting that the cones $\Sos_{X,2j+2k}$ and $f \cdot \Sos_{X,2k}$ are
  well-separated is an open condition in the Euclidean topology on the element
  $f \in R_{2j}$.  Hence, we may assume that the given element $f$ lies in the
  interior of the cone $\Pos_{X,2j}$.  To finish the proof, we prove that,
  under small real perturbations of both $X$ and $f$, the pertinent cones
  continue to be well-separated.

  We first deform the singular Harnack curve $X$ into a smooth Harnack curve
  $X_\varepsilon$.  For brevity, let $H$ denote the very ample divisor
  $(j-1)A_Q$.  Fix a section
  $g_1 \in H^0\bigl(Y_Q, \mathcal{O}_{Y_Q}(H)\bigr)$ defining $X$ on $Y_Q$.
  Since $H$ is very ample, we may choose a section
  $g_2 \in H^0\bigl(Y_Q, \mathcal{O}_{Y_Q}(H)\bigr)$ that does not vanish at
  any solitary point of $X$, so the quotient $g_1/g_2$ is real-valued on
  $Y_Q \setminus \variety(g_2)$ and every solitary point of $X$ is either a
  local maximum or local minimum.  The product of sections defining the
  irreducible torus-invariant divisors determines a section
  $g_3 \in H^0\bigl(Y_Q, \mathcal{O}_{Y_Q}(-K) \bigr)$ because the canonical
  divisor on the toric variety $Y_Q$ is $K = - D_1 - D_2 - \dotsb - D_m$.  As
  the first paragraph in the proof of Proposition~\ref{pro:lower} establishes,
  the solitary points impose independent conditions on the sections of
  $\mathcal{O}_{Y_Q}(H + K)$.  It follows that there exists a section
  $g_4 \in H^0\bigl(Y_Q, \mathcal{O}_{Y_Q}(H+K) \bigr)$ such that the rational
  function $g_3 g_4 / g_2$ has prescribed values at the solitary points of
  $X$.  In particular, we may choose the section $g_4$ so that $g_3 g_4 / g_2$
  is negative at the local minima of the quotient $g_1/g_2$ and is positive at
  the local maxima of the quotient $g_1/g_2$.  For small enough
  $\varepsilon > 0$, we see that the section $g_1 + \varepsilon g_3g_4$
  defines a smooth Harnack curve $X_\varepsilon$ on $Y_Q$ with arithmetic
  genus $\genus$.  Moreover, the sections defining $X_\varepsilon$ and $X$
  have the same degree, so we have $\HF_{X_\varepsilon}(i) = \HF_{X}(i)$ for
  all $i \in \ZZ$.

  To deform the element $f \in \Pos_{X,2j}$, choose a polynomial
  $\tilde{f} \in S_{2j}$ that maps to $f$ under the canonical quotient
  homomorphism, set $e \coloneq \HF_X(2j+2k)$, and fix points
  $p_1, p_2, \dotsc, p_e$ in $X$ for which the linear functionals
  $p_1^*, p_1^*, \dotsc, p_e^*$, defined by point evaluation, form a basis for
  $R_{2j+2k}^*$.  Since the cones $\Sos_{X,2j+2k}$ and $f \cdot \Sos_{X,2k}$
  are well-separated, there exists a linear functional $\ell \in R_{2j+2k}^*$
  satisfying $\ell(h) > 0$ for all nonzero $h \in \Sos_{X,2j+2k}$ and
  $\ell(h) < 0$ for all nonzero $h \in f \cdot \Sos_{X,2k}$.  Hence, there are
  $\lambda_1, \lambda_2, \dotsc, \lambda_e \in \RR$ such that
  $\ell = \lambda_1 p_1^* + \lambda_2 p_2^* + \dotsb + \lambda_e p_e^*$.  By
  choosing affine representatives
  $\tilde{p}_1, \tilde{p}_2, \dotsc, \tilde{p}_e \in \AA^{n+1}$, we obtain
  $\tilde{\ell} \coloneq \lambda_1\tilde{p}_1^* + \lambda_2 \tilde{p}_2^* + \dotsb +
  \lambda_e \tilde{p}_e^*$ in $S_{2j+2k}^*$.  There are two symmetric forms
  associated to the linear functional $\tilde{\ell}$: the first
  $\sigma_{j+k}^*(\tilde{\ell}) \colon S_{j+k} \otimes_\RR S_{j+k} \to \RR$ is
  defined by
  $\tilde{h}_1 \otimes \tilde{h}_2 \mapsto \tilde{\ell}(\tilde{h}_1
  \tilde{h}_2)$ and the second
  $\tau_{j}^*(\tilde{\ell}) \colon S_k \otimes_\RR S_k \to \RR$ is defined by
  $\tilde{h}_1 \otimes \tilde{h}_2 \mapsto \tilde{\ell}(\tilde{f} \tilde{h}_1
  \tilde{h}_2 )$.  The assertion that $\ell$ is a strict separator for the
  cones $\Sos_{X,2j+2k}$ and $f \cdot \Sos_{X,2k}$ is equivalent to saying
  that the symmetric form $\sigma_{j+k}^*(\tilde{\ell})$ is
  positive-semidefinite with
  $\Ker\bigl( \sigma_{j+k}^*(\tilde{\ell}) \bigr) = (I_{X})_{j+k}$ and the
  symmetric form $\tau_{k}^*(\tilde{\ell})$ is negative-semidefinite with
  $\Ker\bigl( \tau_{k}^*(\tilde{\ell}) \bigr) = (I_{X})_{k}$.  To build the
  applicable linear functional on the deformation $X_\varepsilon$, let
  $q_1, q_2, \dotsc, q_e$ denote the points on $X_\varepsilon$ corresponding
  to the fixed points $p_1, p_2, \dotsc, p_e$ on $X$.  Choose affine
  representatives
  $\tilde{q}_1, \tilde{q}_2, \dotsc, \tilde{q}_e \in \AA^{n+1}$ and consider
  the linear functional
  \[
  \tilde{\ell}_{\varepsilon} \coloneq \lambda_1 \tilde{q}_1^* + \lambda_2
  \tilde{q}_2^* + \dotsb + \lambda_e \tilde{q}_e^* \in S_{2j+2k}^* \, .
  \]
  By construction, we have
  $(I_{X_{\varepsilon}})_{j+k} \subseteq \Ker\bigl(
  \sigma_{j+k}^*(\tilde{\ell}_{\varepsilon}) \bigr)$
  and
  $(I_{X_{\varepsilon}})_{k} \subseteq \Ker\bigl(
  \tau_{k}^*(\tilde{\ell}_{\varepsilon}) \bigr)$.
  For sufficiently small $\varepsilon > 0$, the symmetric forms
  $\sigma_{j+k}^*(\tilde{\ell}_\varepsilon)$ and
  $\tau_{k}^*(\tilde{\ell}_\varepsilon)$ are small perturbations of
  $\sigma_{j+k}^*(\tilde{\ell})$ and $\tau_{k}^*(\tilde{\ell})$ respectively.
  The rank of a symmetric form is lower semicontinuous, so we have both
  $\rank \bigl( \sigma_{j+k}^*(\tilde{\ell}_\varepsilon) \bigr) \geq \rank
  \bigl( \sigma_{j+k}^*(\tilde{\ell}) \bigr)$
  and
  $\rank \bigl( \tau_{k}^*(\tilde{\ell}_\varepsilon) \bigr) \geq \rank \bigl(
  \tau_{k}^*(\tilde{\ell}) \bigr)$.
  Because $\HF_{X_\varepsilon}(k) = \HF_{X}(k)$ and
  $\HF_{X_\varepsilon}(j+k) = \HF_{X}(j+k)$, it follows that
  $\Ker\bigl( \sigma_{j+k}^*(\tilde{\ell}_{\varepsilon}) \bigr) =
  (I_{X_{\varepsilon}})_{j+k}$
  and
  $\Ker\bigl( \tau_{k}^*(\tilde{\ell}_{\varepsilon}) \bigr) =
  (I_{X_{\varepsilon}})_{k}$.
  In addition, being positive-semidefinite or negative-semidefinite is an open
  condition in the Euclidean topology, so the symmetric form
  $\sigma_{j+k}^*(\tilde{\ell}_\varepsilon)$ is positive-semidefinite and
  symmetric form $\tau_{k}^*(\tilde{\ell}_\varepsilon)$ is
  negative-semidefinite.  If $f_\varepsilon$ denotes the image of $\tilde{f}$
  under the canonical quotient map from $S$ to $\ZZ$-graded coordinate ring of
  $X_\varepsilon$, then we conclude that
  $\ell_\varepsilon \coloneq \lambda_1 q_1^* + \lambda_2 q_2^* + \dotsb + \lambda_e
  q_e^*$
  is a strict separator for the cones $\Sos_{X_\varepsilon,2j+2k}$ and
  $f_\varepsilon \cdot \Sos_{X_\varepsilon,2j+2k}$.
\end{proof}

\begin{remark}
  Although the smooth curves constructed in the proof of
  Theorem~\ref{thm:lower} have the maximal number of ovals, this is not
  necessary.  By choosing the section $g_4$ so that $g_3 g_4/ g_2$ is positive
  at some local minima, or negative at some local maxima, of the quotient
  $g_1/g_3$, we can obtain smooth curves for which the number of ovals is
  anywhere between $1$ and one more than the genus.  In particular,
  Theorem~\ref{thm:lower} is remarkably insensitive to the topology of the
  real projective curve.
\end{remark}

\begin{remark}
  By applying the perturbation techniques from the proof of
  Theorem~\ref{thm:lower} to the tricuspidal curve in
  Example~\ref{exa:deltoid}, we see that there are smooth planar curves for
  which the bound in Example~\ref{exa:planar} is tight.
\end{remark}

For the smooth curves created in the proof of Theorem~\ref{thm:lower}, both
the degree and genus can be expressed as a function of the parameter $j$.
From these expressions, we see that, for all $j \geq 2$, there are smooth
curves for which Theorem~\ref{thm:lower} is an exact counterpart to
Corollary~\ref{cor:hilbertRegularity}.

\begin{example}[Curves with sharp bounds]
  \label{exa:sharp}
  Let $Q \subset \RR^2$ be a smooth convex lattice polygon with an interior
  lattice point.  Hence, we obtain a smooth toric variety $Y_Q \subset \PP^n$
  embedded by the very ample line bundle $A_Q$.  The Ehrhart polynomial of $Q$
  equals the Hilbert polynomial of $Y_Q \subset \PP^n$; see Proposition~9.4.3
  in \cite{CLS}.  If $\Area(Q)$ denotes the standard Euclidean area of the
  polygon $Q$ and $\abs{\smash{\partial Q \cap \ZZ^2}}$ counts the number of
  lattice points on its boundary $\partial Q$, then it follows that
  $\HP_{Y_Q}(i) = \Area(Q) i^2 + \tfrac{1}{2} \abs{\smash{\partial Q \cap
      \ZZ^2}} i + 1$; see Proposition~10.5.6 in \cite{CLS}.

  Fix an integer $j$ with $j \geq 2$.  Since the smooth curves $X$ appearing
  in the proof of Theorem~\ref{thm:lower} are defined by a section in
  $H^0\bigl( Y_Q, \mathcal{O}_{Y_Q} \bigl( (j-1)A_Q \bigr) \bigr)$, we have
  \begin{align*}
    \HP_{X}(i) 
    &= \HP_{Y_Q}(i) - \HP_{Y_Q}\bigl(i-(j-1) \bigr) =
      2\Area(Q)(j-1) i + \tfrac{1}{2} \abs{\smash{\partial Q \cap \ZZ^2}}
      (j-1) - \Area(Q)(j-1)^2 \, , 
  \end{align*}
  so the degree and genus of $X$ are $2\Area(Q)(j-1)$ and $\HP_{Y_Q}(1-j)$
  respectively.  Amusingly, we have $\deg(X) = (j-1)\deg(Y_Q)$ and the genus
  equals the number of interior lattice points in the dilate $(j-1)Q$; see
  Theorem~9.4.2 in \cite{CLS}.  In addition, the equation for $\HP_{X}(i)$
  implies that $\rr(X) = j - 1 - m$ where $m$ is the largest nonnegative
  integer such that the dilate $mQ$ does not contain any interior lattice
  points.  Since a smooth polytope has at least three vertices, we have
  $3 \leq \abs{\smash{\partial Q \cup \ZZ^2}}$,
  $1 < \tfrac{1}{2} \abs{\smash{\partial Q \cap \ZZ^2}}(j-1)$, and
  \begin{align*}
    \left\lceil \frac{2 \genus}{d} \right\rceil
    &= \left\lceil \frac{\Area(Q)(j-1)^2 - \tfrac{1}{2}
      \abs{\smash{\partial Q \cap \ZZ^2}}(j-1) +1}{\Area(Q)(j-1)}
      \right\rceil \\
    &\leq (j-1) + \left\lceil \frac{1- \tfrac{1}{2}
      \abs{\smash{\partial Q \cap \ZZ^2}}(j-1)}{\Area(Q)(j-1)} \right\rceil
      \leq j-1 \, .
  \end{align*}
  As $Q$ has at least one interior lattice point, we also have
  $1 \leq \HP_{Y_Q}(-1) = \Area(Q) - \tfrac{1}{2} \abs{\smash{\partial Q \cap
      \ZZ^2}} + 1$, $1 < \Area(Q)$, and
  \[
  \left\lfloor \frac{2 \genus}{d} \right\rfloor \geq (j-1) - \left\lceil
    \frac{ \tfrac{1}{2} \abs{\smash{\partial Q \cap \ZZ^2}}}{\Area(Q)}
  \right\rceil + \left\lfloor \frac{1}{\Area(Q)(j-1)} \right\rfloor = j-2 \, .
  \]
  Therefore, Theorem~\ref{thm:lower} proves that, for all $j \geq 2$, there
  exist smooth curves $X \subset \PP^n$ and elements $f \in \Pos_{X,2j}$ such
  that, for all $k < j - 1$, the cones $\Sos_{X,2j+2k}$ and
  $f \cdot \Sos_{X,2k}$ are well-separated.  Conversely,
  Corollary~\ref{cor:hilbertRegularity} establishes that, for all
  $f \in \Pos_{X,2j}$ and for all $k \geq j-1$, the cones $\Sos_{X,2j+2k}$ and
  $f \cdot \Sos_{X,2k}$ are not well-separated.  \hfill $\diamond$
\end{example}

To be comprehensive, we also consider the smooth convex lattice polygons
without an interior lattice point.  From the classification of smooth toric
surfaces (see Theorem~10.4.3 in \cite{CLS}), we see that the polytopes omitted
by Example~\ref{exa:sharp} correspond to Hirzebruch surfaces and the
projective plane.  Using similar techniques to analyze these polytopes, we
produce curves with sharp bounds contained in slightly smaller projective
spaces.

\begin{example}[Sharp bound for curves on Hirzebruch surfaces]
  \label{exa:Hirzebruch}
  For all $r,s \in \NN$, consider the smooth lattice polygon
  $Q \coloneq \conv\{ (0,0), (s+1,0), (r+s+1,1), (0,1) \} \subset \RR^2$.  Since
  $\abs{\smash{Q \cap \ZZ^2}} = r + 2 s + 4$, we obtain, for all $n \geq 3$, a
  Hirzebruch surface
  $Y_Q = \PP\bigl( \mathcal{O}_{\PP^1} \oplus \mathcal{O}_{\PP^1}(r) \bigr)
  \subset \PP^n$
  embedded by the very ample line bundle $A_Q$.  Fix an integer $j$ with
  $j \geq 2$.  Because we have $\Area(Q) = \frac{1}{2}r + s + 1$ and
  $\abs{\smash{\partial Q \cap \ZZ^2}} = r + 2s + 4$, the calculations in
  Example~\ref{exa:sharp} establish that, for the relevant curves
  $X \subset Y_Q$, we have
  \[
  \frac{2 \genus}{d} = \frac{\bigl(\tfrac{1}{2} r  + s + 1\bigr)(j-1)^2 -
    \frac{1}{2}(r + 2s + 4)(j-1) +1}{\bigl(\tfrac{1}{2}r + s + 1 \bigr)(j-1)}
  = j-2 + \frac{2-j}{\bigl(\tfrac{1}{2}r + s + 1 \bigr)(j-1)} \, ,
  \]
  and $j-3 < \bigl\lceil \frac{2 \genus}{d} \bigr\rceil \leq j-2$.  In
  addition, we have $\rr(X) = j-2$.  Therefore, Theorem~\ref{thm:lower} proves
  that, for all $n \geq 3$ and for all $j \geq 2$, there exist smooth curves
  $X \subset \PP^n$ and elements $f \in \Pos_{X,2j}$ such that, for all
  $k < j - 2$, the cones $\Sos_{X,2j+2k}$ and $f \cdot \Sos_{X,2k}$ are
  well-separated.  Conversely, Corollary~\ref{cor:hilbertRegularity}
  establishes that, for all $f \in \Pos_{X,2j}$ and for all $k \geq j-2$, the
  cones $\Sos_{X,2j+2k}$ and $f \cdot \Sos_{X,2k}$ are not well-separated.
  \hfill $\diamond$
\end{example}

\begin{example}[Sharp bounds for planar curves]
  \label{exa:PP2}
  Let $Q \coloneq \conv\{ (0,0), (1,0), (0,1) \} \subset \RR^2$ be the standard
  simplex.  Since $\abs{\smash{Q \cap \ZZ^2}} = 3$, we have
  $Y_Q = \PP^2 \subseteq \PP^2$ embedded by the very ample line bundle
  $A_Q = \mathcal{O}_{\PP^2}(1)$.  Fix an integer $j$ with $j \geq 2$.
  Because we have $\Area(Q) = \frac{1}{2}$ and
  $\abs{\smash{\partial Q \cap \ZZ^2}} = 3$, the calculations in
  Example~\ref{exa:sharp} establish that, for the relevant curves
  $X \subset Y_Q$, we have
  \[
  \frac{2 \genus}{d} = \frac{\tfrac{1}{2}(j-1)^2 -
    \frac{3}{2}(j-1) +1}{\tfrac{1}{2}(j-1)} = j-4 +
  \frac{2}{j-1} \, .
  \]
  When $j \geq 3$, we obtain
  $j-4 < \bigl\lceil \frac{2 \genus}{d} \bigr\rceil \leq j-3$ and, when
  $j = 2$, we have $\frac{2 \genus}{d} = 0$.  In addition, we have
  $\rr(X) = j-3$.  Therefore, Theorem~\ref{thm:lower} proves that, for all
  $j \geq 2$, there exist smooth curves $X \subset \PP^2$ and elements
  $f \in \Pos_{X,2j}$ such that, for all $0 \leq k < j - 3$, the cones
  $\Sos_{X,2j+2k}$ and $f \cdot \Sos_{X,2k}$ are well-separated.  Conversely,
  Corollary~\ref{cor:hilbertRegularity} establishes that, for all
  $f \in \Pos_{X,2j}$ and for all $k \geq \max\{j-3,0\}$, the cones
  $\Sos_{X,2j+2k}$ and $f \cdot \Sos_{X,2k}$ are not well-separated.  \hfill
  $\diamond$
\end{example}

\begin{example}[Sharp bounds on the Veronese surface]
  \label{exa:Veronese}
  Let $Q \coloneq \conv\{ (0,0), (2,0), (0,2) \} \subset \RR^2$.  Since
  $\abs{\smash{Q \cap \ZZ^2}} = 6$, we obtain the Veronese surface
  $Y_Q \subset \PP^5$ embedded by the very ample line bundle
  $A_Q = \mathcal{O}_{\PP^2}(2)$.  Fix an integer $j$ with $j \geq 2$.
  Because we have $\Area(Q) = 2$ and
  $\abs{\smash{\partial Q \cap \ZZ^2}} = 6$, the calculations in
  Example~\ref{exa:sharp} establish that, for the relevant curves
  $X \subset Y_Q$, we have
  \[
  \frac{2 \genus}{d} = \frac{2(j-1)^2 - 3(j-1) +1}{2(j-1)} = j - 2 +
  \frac{2-j}{2(j-1)} \, ,
  \]
  and $j - 3 < \bigl\lceil \frac{2 \genus}{d} \bigr\rceil \leq j-2$.  In
  addition, we have $\rr(X) = j - 2$.  Therefore, Theorem~\ref{thm:lower}
  proves that, for all $j \geq 2$, there exist smooth curves
  $X \subset \PP^5$ and elements $f \in \Pos_{X,2j}$ such that, for all
  $k < j - 2$, the cones $\Sos_{X,2j+2k}$ and $f \cdot \Sos_{X,2k}$ are
  well-separated.  Conversely, Corollary~\ref{cor:hilbertRegularity}
  establishes that, for all $f \in \Pos_{X,2j}$ and for all $k \geq j-2$, the
  cones $\Sos_{X,2j+2k}$ and $f \cdot \Sos_{X,2k}$ are not well-separated.
  \hfill $\diamond$
\end{example}

\begin{proof}[Proof of Theorem~\ref{thm:main}] 
  Corollary~\ref{cor:hilbertRegularity} proves the first part. If one
  overlooks the parameter $n$, then Theorem~\ref{thm:lower} immediately proves
  the second part.  By combining Example~\ref{exa:Hirzebruch} and
  Example~\ref{exa:PP2}, it follows that the required curves and nonnegative
  elements exist for all $n \geq 2$.
\end{proof}

We end this paper by lifting these degree bounds for strict-separators from
curves to some surfaces.  To accomplish this, we exploit the perturbation
methods used in the proof of Theorem~\ref{thm:lower}.

\begin{proposition}
  \label{pro:lifting}
  Fix a positive integer $j$ and a nonnegative integer $k$.  Let
  $X \subseteq \PP^n $ be an arithmetically Cohen--Macaulay real projective
  variety and let $X'$ be a hypersurface section of $X$ of degree $j$.  If
  there exists an element $f' \in \Pos_{X',2j}$ such that the cones
  $\Sos_{X',2j+2k}$ and $f' \cdot \Sos_{X',2k}$ are well-separated, then there
  exists an element $f \in \Pos_{X, 2j}$ such that the cones $\Sos_{X,2j+2k}$
  and $f \cdot \Sos_{X,2k}$ are also well-separated.
\end{proposition}

\begin{proof}
  We first lift $f'$ to a nonnegative element on $X$.  As observed in the
  proof of Proposition~\ref{thm:lower}, asserting that the cones
  $\Sos_{X',2j+2k}$ and $f' \cdot \Sos_{X',2k}$ are well-separated is an open
  condition in the Euclidean topology on the element $f' \in R_{2j}'$.  Hence,
  we may assume that $f'$ is positive on $X'(\RR)$.  Choose a homogeneous
  polynomial $\tilde{f}' \in S_{2j}$ that maps to $f'$ under the canonical
  quotient homomorphism from $S$ to $R' = S/I_{X'}$.  By hypothesis, $X'$ is a
  hypersurface section of $X$ of degree $j$, so there is a nonzero polynomial
  $h \in S_j$ such that $X' = X \cap \variety(h) \subset \PP^n$.  Moreover, we
  have $I_{X'} = I_X + \langle h \rangle$ because $X$ is arithmetically
  Cohen--Macaulay. Let $\tilde{X} \subseteq \AA^{n+1}(\RR)$ be the affine cone
  of $X$ and let $\SS^n \subset \AA^{n+1}(\RR)$ be the unit sphere.  Since
  $\tilde{f}'$ is positive on $X'(\RR)$, there exists a Euclidean
  neighbourhood $U$ of $\SS^n \cap \variety(h) \subset \SS^n \cap \tilde{X}$
  such that $\tilde{f}'$ is positive on $U$.  On the compact set
  $K \coloneq ( \SS^n \cap \tilde{X} ) \setminus U$, the function $h^2$ is
  positive, so
  \[
  \delta \coloneq \frac{ \inf_K h^2 }{\sup_K \abs{\tilde{f}'}}
  \] 
  is a positive real number.  It follows that, for all
  $\lambda > \frac{1}{\delta}$, the polynomial
  $\tilde{f} \coloneq \tilde{f}' + \lambda h^2$ is positive on $X(\RR)$. Therefore,
  if $f$ denotes the image of $\tilde{f}$ under the canonical quotient
  homomorphism from $S$ to $R = S/I_{X}$, then we deduce that
  $f \in \Pos_{X,2j}$.

  We next deform $X'$ and $f'$.  If $h = \sum_{\abs{u} = 2j} c_{u} x^u$ where
  $u \in \NN^{n+1}$ and $c_{u} \in \RR$, then consider the homogeneous
  polynomial
  $h_{\varepsilon} \coloneq \sum_{\abs{u} = 2j} (c_{u} + \varepsilon_u)x^u$
  with $\abs{\varepsilon_u} < \varepsilon$ created by perturbing the
  coefficients and the corresponding hypersurface section
  $X_{\varepsilon}' \coloneq X \cap \variety(h_\varepsilon) \subset
  \PP^n$. Set $e \coloneq \HF_{X'}(2j+2k)$ and fix points
  $p_1, p_2, \dotsc, p_e$ in $X'$ for which the linear functionals
  $p_1^*, p_1^*, \dotsc, p_e^*$, defined by point evaluation, form a basis for
  $(R'_{2j+2k})^*$.  Since the cones $\Sos_{X',2j+2k}$ and
  $f' \cdot \Sos_{X',2k}$ are well-separated, there exists a linear functional
  $\ell \in (R'_{2j+2k})^*$ satisfying $\ell(g) > 0$ for all nonzero
  $g \in \Sos_{X',2j+2k}$ and $\ell(g) < 0$ for all nonzero
  $g \in f' \cdot \Sos_{X',2k}$.  It follows that there are real numbers
  $\lambda_1, \lambda_2, \dotsc, \lambda_e$ such that
  $\ell = \lambda_1 p_1^* + \lambda_2 p_2^* + \dotsb + \lambda_e p_e^*$.  By
  choosing affine representatives
  $\tilde{p}_1, \tilde{p}_2, \dotsc, \tilde{p}_e \in \AA^{n+1}$, we obtain
  $\tilde{\ell} \coloneq \lambda_1\tilde{p}_1^* + \lambda_2 \tilde{p}_2^* +
  \dotsb + \lambda_e \tilde{p}_e^*$ in $S_{2j+2k}^*$.  As in the proof of
  Theorem~\ref{thm:lower}, there are two symmetric forms associated to the
  linear functional $\tilde{\ell}$: the first
  $\sigma_{j+k}^*(\tilde{\ell}) \colon S_{j+k} \otimes_\RR S_{j+k} \to \RR$ is
  defined by
  $\tilde{g}_1 \otimes \tilde{g}_2 \mapsto \tilde{\ell}(\tilde{g}_1
  \tilde{g}_2)$ and the second
  $\tau_{j}^*(\tilde{\ell}) \colon S_k \otimes_\RR S_k \to \RR$ is defined by
  $\tilde{g}_1 \otimes \tilde{g}_2 \mapsto \tilde{\ell}(\tilde{f} \tilde{g}_1
  \tilde{g}_2 )$.  The assertion that $\ell$ is a strict separator for the
  cones $\Sos_{X',2j+2k}$ and $f' \cdot \Sos_{X',2k}$ is equivalent to saying
  that symmetric form $\sigma_{j+k}^*(\tilde{\ell})$ is positive-semidefinite
  with $\Ker\bigl( \sigma_{j+k}^*(\tilde{\ell}) \bigr) = (I_{X'})_{j+k}$ and
  symmetric form $\tau_{k}^*(\tilde{\ell})$ is negative-semidefinite with
  $\Ker\bigl( \tau_{k}^*(\tilde{\ell}) \bigr) = (I_{X'})_{k}$.  To build the
  applicable linear functional on a deformation $X'_\varepsilon$, let
  $q_1, q_2, \dotsc, q_e$ denote the points on $X'_\varepsilon$ corresponding
  to the fixed points $p_1, p_2, \dotsc, p_e$ on $X$.  Choose affine
  representatives
  $\tilde{q}_1, \tilde{q}_2, \dotsc, \tilde{q}_e \in \AA^{n+1}$ and consider
  the linear functional
  $\tilde{\ell}_{\varepsilon} \coloneq \lambda_1 \tilde{q}_1^* + \lambda_2
  \tilde{q}_2^* + \dotsb + \lambda_e \tilde{q}_e^*$ in $S_{2j+2k}^*$.  By
  construction, we have
  $(I_{X_{\varepsilon}})_{j+k} \subseteq \Ker\bigl(
  \sigma_{j+k}^*(\tilde{\ell}_{\varepsilon}) \bigr)$ and
  $(I_{X_{\varepsilon}})_{k} \subseteq \Ker\bigl(
  \tau_{k}^*(\tilde{\ell}_{\varepsilon}) \bigr)$.  For sufficiently small
  $\varepsilon > 0$, the symmetric forms
  $\sigma_{j+k}^*(\tilde{\ell}_\varepsilon)$ and
  $\tau_{k}^*(\tilde{\ell}_\varepsilon)$ are small perturbations of
  $\sigma_{j+k}^*(\tilde{\ell})$ and $\tau_{k}^*(\tilde{\ell})$ respectively.
  The rank of a symmetric form is lower semicontinuous, so we have
  $\rank \bigl( \sigma_{j+k}^*(\tilde{\ell}_\varepsilon) \bigr) \geq \rank
  \bigl( \sigma_{j+k}^*(\tilde{\ell}) \bigr)$ and
  $\rank \bigl( \tau_{k}^*(\tilde{\ell}_\varepsilon) \bigr) \geq \rank \bigl(
  \tau_{k}^*(\tilde{\ell}) \bigr)$.  It follows that
  $\Ker\bigl( \sigma_{j+k}^*(\tilde{\ell}_{\varepsilon}) \bigr) =
  (I_{X'_{\varepsilon}})_{j+k}$ and
  $\Ker\bigl( \tau_{k}^*(\tilde{\ell}_{\varepsilon}) \bigr) =
  (I_{X'_{\varepsilon}})_{k}$ because $\HF_{X'_\varepsilon}(k) = \HF_{X'}(k)$
  and $\HF_{X'_\varepsilon}(j+k) = \HF_{X'}(j+k)$.  In addition, being
  positive-semidefinite or negative-semidefinite is an open condition in the
  Euclidean topology, so the symmetric form
  $\sigma_{j+k}^*(\tilde{\ell}_\varepsilon)$ is positive-semidefinite and
  symmetric form $\tau_{k}^*(\tilde{\ell}_\varepsilon)$ is
  negative-semidefinite.  If $f'_\varepsilon$ denotes the image of $\tilde{f}$
  under the canonical quotient map from $S$ to
  $R'_{\varepsilon} = S/(I_X + \langle h_\varepsilon \rangle)$, then we
  conclude that the linear functional
  $\ell_\varepsilon \coloneq \lambda_1 q_1^* + \lambda_2 q_2^* + \dotsb +
  \lambda_e q_e^*$ is a strict separator for the cones
  $\Sos_{X'_\varepsilon,2j+2k}$ and
  $f'_\varepsilon \cdot \Sos_{X'_\varepsilon,2j+2k}$.

  Lastly, suppose that there exists a nonzero $g \in \Sos_{X,2k}$ such that
  $fg \in \Sos_{X,2j+2k}$.  By construction, the nonnegative element $f$
  restricts to $f'_\varepsilon$ and the cones $\Sos_{X'_\varepsilon,2j+2k}$
  and $f'_\varepsilon \cdot \Sos_{X'_\varepsilon,2j+2k}$ are well-separated,
  so the multiplier $g$ restricts to $0$ on $X'_\varepsilon$.  Equivalently,
  if $g'_\varepsilon$ denotes the image of $g$ under the canonical quotient
  map from $R$ to $R'_\varepsilon = R/\langle h_\varepsilon \rangle$, then we
  have $g'_\varepsilon \in \langle h_\varepsilon \rangle$.  Since this holds
  for all sufficiently small $\varepsilon \geq 0$, we see that $g = 0$ in $R$
  which is a contradiction.  Therefore, the cones $\Sos_{X,2j+2k}$ and
  $f \cdot \Sos_{X,2k}$ are also well-separated.
\end{proof}

The final two examples illustrate this proposition and provide explicit degree
bounds on strict-separators on some smooth toric surfaces.  Unlike for curves,
our techniques do not typically prove that these degree bounds are sharp.
However, for the classical case of ternary octics, we do obtain tight degree
bounds for the existence of sum-of-squares multipliers.

\begin{example}[Strict-separators on toric surfaces of minimal degree]
  \label{exa:strictMinimalDegree}
  Let $X$ be a toric surface of minimal degree.  By combining
  Example~\ref{exa:Hirzebruch} or Example~\ref{exa:Veronese} with
  Proposition~\ref{pro:lifting}, it follows that, for all $j \geq 2$, there
  exist elements $f \in \Pos_{X,2j}$ such that, for all $k < j-2$, the cones
  $\Sos_{X,2j+2k}$ and $f \cdot \Sos_{X,2k}$ are well-separated.  In
  constrast, Example~\ref{exa:minimalSurface} only establishes that, for all
  $f \in \Pos_{X,2j}$, the cones $\Sos_{X,j^2+j}$ and $f \cdot \Sos_{X,j^2-j}$
  are not well-separated, so there is a gap between our bounds.  Since
  Example~\ref{exa:minimalSurface} also proves that, for all
  $f \in \Pos_{X,2j}$, the cones $\Sos_{X,4j-4}$ and $f \cdot \Pos_{X,2(j-1)}$
  are not well-separated, there is even a gap when we consider all nonnegative
  multipliers. \hfill $\diamond$
\end{example}

\begin{proof}[Proof of Theorem~\ref{thm:surface}]
  Example~\ref{exa:minimalSurface} proves the first half and
  Example~\ref{exa:strictMinimalDegree} proves the second.
\end{proof}

\begin{example}[Strict-separators on the projective plane]
  \label{exa:strictPP^2}
  Let $Q \coloneq \conv\{ (0,0), (1,0), (0,1) \} \subset \RR^2$ and let
  $\PP^2 = Y_Q \subseteq \PP^2$ be embedded by the very ample line bundle
  $A_Q = \mathcal{O}_{\PP^2}(1)$.  By combining Example~\ref{exa:PP2} and
  Proposition~\ref{pro:lifting}, it follows that, for all $j \geq 2$, there
  exist elements $f \in \Pos_{\PP^2,2j}$ such that, for all $k < j-3$, the
  cones $\Sos_{\PP^2,2j+2k}$ and $f \cdot \Sos_{\PP^2,2k}$ are well-separated.
  Example~\ref{exa:formsOnPP2} shows that, for all $f \in \Pos_{\PP^2, 8}$,
  the cones $\Sos_{\PP^2, 12}$ and $f \cdot \Sos_{\PP^2,4}$ are not
  well-separated, so this degree bound for strict-separators on $\PP^2$ is
  sharp when $j = 4$. \hfill $\diamond$
\end{example}

%%----------------------------------------------------------------------------
\section*{Acknowledgements}

\noindent
We thank Erwan Brugall\'e, Lionel Lang, and Mike Roth for useful conversations.
The first author was partially supported by an Alfred P.{} Sloan Fellowship,
the Simons Institute, and an NSF CAREER award DMS--1352073; the second author
was partially supported by NSERC; and the third author was partially supported
by the FAPA funds from Universidad de los Andes.

%%%%%%%%%%%%%%%%%%%%%%%%%%%%%%%%%%%%%%%%%%%%%%%%%%%%%%%%%%%%%%%%%%%%%%%%%%%%%%

\raggedright

\end{document}